\documentclass[onefignum,onetabnum]{siamart220329}
\usepackage{amsfonts,amssymb}
 \usepackage{color}
 \usepackage{enumitem}
\newsiamremark{example}{Example}
\crefname{example}{Example}{Examples}
\newcommand{\R}{\mathbb{R}}
\newcommand{\B}{\mathbb{B}}
\newcommand{\Sph}{\mathbb{S}}

\DeclareMathOperator*\dom{dom}
\DeclareMathOperator*\conv{conv}
\DeclareMathOperator*\pos{pos}
\DeclareMathOperator*\gph{gph}
\DeclareMathOperator*\inte{int}

\DeclareMathOperator*\proj{proj}

\DeclareMathOperator*\lip{lip}

\usepackage{lipsum}
\usepackage{amsfonts}
\usepackage{graphicx}
\usepackage{epstopdf}
\usepackage{algorithmic}
\ifpdf
  \DeclareGraphicsExtensions{.eps,.pdf,.png,.jpg}
\else
  \DeclareGraphicsExtensions{.eps}
\fi

\newsiamremark{remark}{Remark}
\newsiamremark{hypothesis}{Hypothesis}
\crefname{hypothesis}{Hypothesis}{Hypotheses}
\newsiamthm{claim}{Claim}

\headers{Relative Lipschitz-like property of parametric systems}{Wenfang Yao and Xiaoqi Yang}

\title{Relative Lipschitz-like property of parametric systems via projectional coderivatives\thanks{Submitted to the editors DATE.
\funding{This work was funded by ....}}}

\author{Wenfang Yao\thanks{Department of Applied Mathematics, The Hong Kong Polytechnic University,  Hong Kong
  (\email{dorothy.yao@polyu.edu.hk}, \email{mayangxq@polyu.edu.hk}).}
\and Xiaoqi Yang\footnotemark[2]}
\usepackage{amsopn}

\ifpdf
\hypersetup{
  pdftitle={Relative Lipschitz-like property of parametric systems via projectional coderivative},
  pdfauthor={Dorothy}
}
\fi

\begin{document}

\maketitle

\begin{abstract}
This paper concerns upper estimates of the projectional coderivative of implicit mappings and corresponding applications on analyzing the relative Lipschitz-like property. Under different constraint qualifications, we provide upper estimates of the projectional coderivative for solution mappings of parametric systems. For the solution mapping of affine variational inequalities, a generalized critical face condition is obtained for sufficiency of its Lipschitz-like property relative to a polyhedral set within its domain under a constraint qualification. The equivalence between the relative Lipschitz-like property and the local inner-semicontinuity for polyhedral multifunctions is also demonstrated. For the solution mapping of linear complementarity problems with a $Q_0$-matrix, we establish a sufficient and necessary condition for the Lipschitz-like property relative to its convex domain via the generalized critical face condition and its combinatorial nature.
\end{abstract}

\begin{keywords}
parametric systems, relative Lipschitz-like property, generalized Mordukhovich criterion, affine variational inequality, linear complementarity problem
\end{keywords}

\begin{MSCcodes}
49J40, 49J53, 49K40, 90C31, 90C33  
\end{MSCcodes}


\section{Introduction}
In this paper, we focus on the relative Lipschitz-like property of the solution mapping of the  following fully parametric system 
\begin{equation}\label{eqn:defofS-ParaSys}
    S(w) := \left\{ x\in \R^n  \mid 0\in G(w, x) + M(w, x)  \right\}
\end{equation}
where $G: \R^{m+n} \rightarrow \R^d$ is a $\mathcal{C}^1 $ mapping, and $M: \R^{m + n} \rightrightarrows \R^d$ is a multifunction with  closed graph. The Lipschitz-like property, originated from \cite{aubin1984lipschitz}, plays a central role in stability analysis. Mordukhovich criterion provides a complete characterization of Lipschitz-like property by using the coderivative, see \cite{mordukhovich1992sensitivity,mordukhovich1994generalized,VaAn}.

An upper estimate, as well as the attainable equality, of the coderivative of the mapping $S$ in \cref{eqn:defofS-ParaSys} was obtained in \cite{LevyMord2004} under some constraint qualifications, and applied to obtain the Lipschitz-like property for the stationary point multifunction of a parametric optimization problem. On the other hand, a critical face condition was developed in \cite{dontchev1996} (see also \cite{dontchev2009implicit}) as sufficient and necessary conditions for the Lipschitz-like property of the solution mapping of an affine variational inequality problem over a polyhedral set. 

Sufficient conditions for Lipschitz-like property of quasi-variational inequalities was obtained in \cite{mordukhovich2007coderivative} by developing new coderivative calculus for special compositions. Sufficient conditions for the Lipschitz-like property of implicit multifunctions (that is, $G$ in \cref{eqn:defofS-ParaSys} vanishes) was established in \cite{GfrOut2016} by using a directional limiting coderivative, which was introduced in \cite{gfrerer2013directional} (with a slightly different form in  \cite{ginchevdirectionally}). A formula was also derived for computing the directional limiting coderivative of the normal-cone map with a polyhedral set, which matches the well known critical face
condition framework in \cite{dontchev1996}. Determinantal conditions for the existence of a single-valued, Lipschitz continuous, piecewise-affine solution trajectory to an affine variational inequality subject
to both a canonical parameter and a perturbed normal cone of a polyhedral convex cone was obtained in \cite{lu2008variational}. Further extension of critical face condition to the Lipschitz-like property of solution mapping to a generalized equation can also be found in \cite{gfrerer2020aubin} under the framework of \cite{mordukhovich2007coderivative}. Characterizations of Lipschitz-like property for the solution mapping of linear semi-infinite and infinite systems were obtained in \cite{canovas2010variational} by developing a Mordukhovich criterion in an arbitrary Banach space based on generalized differentiation.
For more investigation on parametric optimization problems, see monographs by \cite{bonnans2013perturbation,dontchev2009implicit,ioffe2017variational,klatte2006nonsmooth}.

The coderivative of a normal cone mapping was computed and applied to give some sufficient condition for the Lipschitz-like property of the stationary point multifunction of minimizing a quadratic function with a ball constraint in \cite{lee2014coderivatives}. In \cite{Huyen2019,Huyen2016}, the condition for coderivative equality in \cite{LevyMord2004} was applied to study the Lipschitz-like property and the Robinson metric regularity of a parametric affine constraint system with a closed set under full perturbations.
See also a recent monograph \cite{lee2005quadratic} on the study of stability of parametric quadratic programs and affine variational inequalities. 

On the application side of stability theory, the calmness of the solution mapping of the Lasso relative to the positive half-line was established in \cite{bickeletal2009,candestao2005} (see also \cite{huetal2017}  for the relative $q$-order calmness for the $\ell_q$-regularization problem  ($0< q \leq 1$).) As noted in \cite{gfrerer2013directional}, in order to guarantee some stationarity conditions, one may only need a regular behavior of the constraint systems with respect to one single critical direction, not on the whole space. These applications of stability properties are concerned with the case where the reference point may lie on the boundary of the domain or a set under consideration. In the case of the Lipschitz-like property, the well-known Mordukhovich criterion is not applicable. Recently by virtue of directional limiting coderivative of the normal-cone mapping and a variant of critical face condition, sufficient conditions  were obtained in \cite{Benko2020} for the Lipschitz-like property relative
to a closed set for the solution map of a class of parameterized variational systems. 
By employing the projection of the normal cone of the restricted graph of a multifunction on the product of the tangent cone of the concerned closed set and the solution space, a projectional coderivative was introduced and applied in \cite{Meng2020} to derive a generalized Mordukhovich criterion for the Lipschitz-like property relative to a closed and convex set. 

In this paper we begin by giving a comparison between the sufficient conditions for relative Lipschitz-like property in  \cite{Benko2020,Meng2020}. While an example was given to illustrate that the sufficient condition of the generalized Mordukhovich criterion in \cite{Meng2020} holds but the sufficient condition in \cite{Benko2020} does not hold, we further show that the sufficient condition of the relative Lipschitz-like property for an explicit multifunction in \cite{Benko2020} implies the one in \cite{Meng2020} when the relative set is closed and convex.
We then focus on the investigation of projectional coderivative and Lipschitz-like property of fully parametric systems. We first obtain an upper estimate of the projectional coderivative of the solution mapping (\ref{eqn:defofS-ParaSys}) by posing different constraint qualifications. Following \cite{LevyMord2004} and using the definition of the projectional coderivative, we calculate the outer limit of the projection of the normal cone of the restricted graph to the tangent cone of the relative set. The upper estimate becomes tight under a regularity condition and the relative set being a manifold. We also obtain some upper estimates of the projectional coderivative of (\ref{eqn:defofS-ParaSys}) for the two special cases (i) $G(w,x)$ vanishes and (ii) $M(w,x) = M(x)$ and the relative set is equal to the dom$\ S$. 

The obtained upper estimates of the projectional coderivative of $S(x)$ are applied to an affine variational inequality (AVI), where $G(q,x) = q + Mx$ and $M(x) = N_C(x)$ with $M$ being an $n \times n$ matrix and $C$ a polyhedron. Some upper estimates of the projectional coderivative of the solution mapping of the AVI are obtained  under different constraint qualifications and regularity conditions. 
The upper estimate is tight if either the graph of the normal cone $N_C$ and the relative set are both regular or the relative set is equal to the domain of $S$. For studying the relative Lipschitz-like characterization of AVI, we consider that the relative set is a polyhedral subset of the domain. We obtain a sufficient condition for $S$ to be relative Lipschitz-like under a constraint qualification and represent this condition in the form of critical face condition as in \cite{dontchev1996}. This sufficient condition becomes necessary if the graph of the normal cone $N_C$ is regular. We also obtain a characterization of relative Lipschitz-like property by virtue of a relative inner semicontinuity around a point. We will apply the obtained results for AVI to a linear complementarity problem (LCP) with $Q_0$ matrix. It is known that the domain of LCP with $Q_0$ matrix is a polyhedron, see \cite{LCP}. We thus consider this domain as the relative set. Therefore the estimate of the projectional coderivative is tight and the constraint qualification required for the AVI holds automatically so that we are able to obtain sufficient and necessary for the relative Lipschitz-like property for LCP with $Q_0$ matrix. We represent this condition in terms of a normal cone of the complementarity conditions of the LCP. Furthermore we apply this condition to evaluate the relative Lipschitz-like modulus and represent it as a computable form.

The organization of the paper is as follows. \Cref{sect:Prelim} introduces the standard notations and tools.  \Cref{sect:ParaSys} presents upper estimates of projectional coderivatives of parametric systems. \Cref{sect:AVI} characterizes the Lipschitz-like property of the solution mapping of affine variational inequalities relative to a polyhedral set. \Cref{sect:LCP} gives a sufficient and necessary condition of the Lipschitz-like property of the solution mapping of linear complementarity problem with $Q_0$ matrix relative to its domain.

\section{Preliminaries}\label{sect:Prelim}
In this section, we review some notations and preliminary results that will be used in the following sections. The notations adopted are standard in variational analysis. Most of them can be found in monographs \cite{mordukhovich2006variational} and \cite{VaAn}.

The norm and scalar product of an Euclidean space $\R^n$ are denoted as  $\| \cdot \| $ and $\langle \cdot, \cdot \rangle$ respectively. The symbol $\B_r (x) $ stands for the closed unit ball with radius $r > 0 $ centered at $x$ and $\B := \B_1(0)$. By $A^* $ we denote the transpose of matrix $A$ and also the adjoint operator of linear operator $A$  and by $A^{-1}$ we denote the inverse mapping of $A$. 

For a vector $v \in \R^n$ we denote $[v] := \{ \lambda v \mid \lambda \in \R \}$ the linear subspace generated by $v$. For a nonempty set $C\subseteq \R^n$, the interior, the convex hull and the positive hull of $C$ are denoted respectively by $\inte C$, $\conv C$  and $\pos C$. The orthogonal complement $C^\perp$ and the polar cone $C^*$ 
are defined respectively by
\begin{align*}
    & C^\perp  :=\{v\in \R^n\mid \langle v, x\rangle=0, \ \forall x\in C\}, \\
    & C^* :=\{v\in \R^n\mid \langle v, x\rangle\leq 0, \ \forall x\in C\}. 
  \end{align*}
The distance from $x$ to $C$ is defined by $d(x,C):=\inf_{y\in C}||y-x||.$
The projection mapping $\proj_C$ is defined by ${\rm proj}_C (x):=\{y\in C\mid d(x,y) =d(x,C)\}.$
For a set $X\subset \R^n$, we denote the projection of $X$ onto $C$  by
$$
{\rm proj}_C X:=\{y\in C\mid \exists x\in X,\,\mbox{with}\, d(x,y)=d(x,C)\}.
$$
If $C=\emptyset$, by convention we set that $d(x, C):=+\infty$, ${\rm proj}_C(x):=\emptyset$, and ${\rm proj}_C X:=\emptyset$.

Let $x\in C$. We use $T_C(x)$ to denote the {\it tangent/contingent cone} to $C$ at $x$, i.e. $w\in T_C(x)$ if there exist sequences $t_k\searrow 0$ and $\{w_k\}\subset \R^n$ with $w_k\rightarrow w$ and $x+ t_k w_k\in C , \  \forall k$. 
The {\it regular/Fr\'{e}chet  normal cone}, $\widehat{N}_C(x)$, is the polar cone of $T_C(x)$, defined by 
$$\widehat{N}_C(x) = \left\{ v\in \R^n \ \left|  \ \limsup_{x'\xrightarrow[x'\neq x]{C}x}  \frac{\langle v, x'-x \rangle}{\|x'- x\|} \leq 0 \right. \right\}.  $$
The {\it(basic/limiting/Mordukhovich) normal cone} to $C$ at $x$, $N_C(x)$, is defined via the outer limit of $\widehat{N}_C$ 
as 
$$N_C(x) 
:= \left\{ v \in \R^n \mid \exists \text{ sequences } x_k \xrightarrow[]{C} x,  \ v_k \rightarrow v, \ v_k \in \widehat{N}_C(x_k) ,  \forall k\right\}. $$

We say that $C$ is {\it locally closed} at a point $x\in C$ if $C\cap U$ is closed for some closed neighborhood $U \in \mathcal{N}(x)$. 
$C$ is said to be {\it regular at} $x$ in the sense of Clarke if it is locally closed at $x$ and $\widehat{N}_C(x)=N_C(x)$. $C$ is {\it regular around} $x$ if there exists a neighborhood of $x$ such that $C$ is regular at every point in it. 
For any $x \notin C$, we set by convention $T_C (x) = \emptyset, \ N_C(x)=\emptyset, \  \widehat{N}_C(x) = \emptyset$.

Let $C\subset \R^n$ be a nonempty convex set.  A {\it face} of $C$ is a convex subset $C'$ of $C$ such that every closed line segment in $C$ with a relative interior point in $C'$ has both endpoints in $C'$. 
See the book \cite{CvxAn} for more details. Moreover, we say $C$ is a {\it polyhedral} set if it can be expressed as the intersection of a finite number of closed half-spaces of hyperplanes (and therefore convex). 

For a multifunction $S:\R^n\rightrightarrows \R^m$, we denote by
$\gph S:=\{(x, u)\mid u\in S(x)\}$ and $\dom S:=\{x\mid S(x)\not=\emptyset\}$ the {\it graph} and the {\it domain} of $S$, respectively. 

For a set $X\subset \R^n$, we denote by
  \begin{equation*} \label{restrX}
  S|_X(x):=  S(x) \ \mbox{if } x\in X; \
    \emptyset \ \mbox{if } x\not\in X,
  \end{equation*}
  the {\it restricted mapping} of $S$ on $X$. It is clear to see that $\gph S|_X=\gph S\cap (X\times \R^m)$, $\dom  S|_X=X\cap \dom S$ 
and also
 \begin{equation*} \label{restrXlimit}
  \displaystyle \limsup_{\tiny x\xrightarrow[]{X}\bar{x}} S(x) = \limsup_{\tiny x\rightarrow  \bar{x}} S|_X(x).
 \end{equation*}
\begin{definition}[Outer semicontinuity, {\cite[Definition 5.4]{VaAn}}] 
A multifunction $S: \R^n \rightrightarrows \R^m$ is outer semicontinuous (osc) at $\bar{x}$ if
$ \limsup_{\tiny x\rightarrow \bar{x}}S(x) = S(\bar{x}).$
\end{definition}

\begin{definition}[local boundedness relative to a set, {\cite[Page 162]{VaAn}}]
For a multifunction $S: \R^n \rightrightarrows \R^m$, a closed set $X \subset \R^n$ and a given point $\bar{x} \in X$, if for some neighborhood $V\in \mathcal{N}(\bar{x})$, $S(V\cap X)$ is bounded, we say $S$ is locally bounded relative to $X$ at $\bar{x}$. Such definition is equivalent to local boundedness of $S|_X$ at $\bar{x}$.
\end{definition}

Next we present the definition of relative Lipschitz-like property of a multifunction.
\begin{definition}[Lipschitz-like property relative to a set, {\cite[Definition 9.36]{VaAn}}]\label{def-lip-like}
A multifunction $S: \R^n\rightrightarrows  \R^m$ is Lipschitz-like relative to $X$ at $\bar{x}$ for $\bar{u}$, where $\bar{x}\in X$ and $\bar{u}\in S(\bar{x})$, if $\gph S$ is locally closed at $(\bar{x}, \bar{u})$ and there are neighborhoods $V\in \mathcal{N}(\bar{x})$, $W\in \mathcal{N}(\bar{u})$, and a constant $\kappa \in \R_+$ such that
\begin{equation}\label{Ap-def}
S(x')\cap W\subset S(x)+\kappa\|x'-x\|\B\quad \forall x,x'\in X\cap V.
\end{equation}
The graphical modulus of $S$ relative to $X$ at $\bar{x}$ for $\bar{u}$ is defined as
\begin{equation*}
    \begin{aligned}
       {\lip}_X S(\bar{x}\mid\bar{u}):=\inf\{\kappa\geq  0 & \mid  \exists V\in \mathcal{N}(\bar{x}), W\in \mathcal{N}(\bar{u}),\;\mbox{such that}\\[0.3cm]
&S(x')\cap W\subset S(x)+\kappa\|x'-x\|\B, \ \forall x,x'\in X\cap V\;\}.
    \end{aligned}
\end{equation*}
The property with $V$ in place of $X\cap V$ in \cref{Ap-def} is the Lipschitz-like property along with the graphical modulus $\lip S(\bar{x} \mid \bar{u})$. 
\end{definition}

Another notion of Lipschitz continuity mentioning $\bar{u}$ is the {\it local Lipschitz continuity around $(\bar{x}, \bar{u})$}, with the sets $S(x)$, $X\cap V$ being replaced by $S(x) \cap W$ and $X\cap V$ respectively (see \cite{dontchev1996}).
It is also known as {\it truncated Lipschitz continuity}, see \cite[p.165]{dontchev2009implicit}, and is generally stronger than the Lipschitz-like property. 

Now we recall a projectional coderivative and a complete characterization of Lipschitz-like property relative to a closed and convex set.
\begin{definition}[{\cite[Definition 2.2]{Meng2020}}]\label{def:ProjCode}
  The projectional coderivative,  $D^*_X S(\bar{x} \mid \bar{u}) : \R^m \rightrightarrows \R^n$  of multifunction $S: \R^n \rightrightarrows \R^m$ at $\bar{x} \in X\subseteq \R^n$ for any $\bar{u} \in S(\bar{x})$ with respect to $X$ is defined as
  \begin{equation*}
  t^*\in D^*_{X}S(\bar{x}\mid \bar{u})(u^*)\Longleftrightarrow (t^*, -u^*)\in \limsup_{\tiny (x,u)\xrightarrow[]{\gph S|_X}(\bar{x}, \bar{u})} {\rm proj}_{T_X(x)\times \R^m}N_{\gph S|_X}(x,u).
  \end{equation*}
\end{definition}

A connection between projectional coderivative and coderivative is shown in the following corollary. 
\begin{corollary}\label{lm:ProjCodeInverseAt0-Incl}
  For a multifunction $S: \R^n \rightrightarrows \R^m$, and a closed set $X \subseteq \R^n$, for any $(\bar{x}, \bar{u})\in \gph S|_X$,
  \begin{equation*}
    D^*S|_X (\bar{x} \mid \bar{u})^{-1} (0) \subseteq D^*_X S (\bar{x} \mid \bar{u})^{-1}(0).
  \end{equation*}
\end{corollary}
\begin{proof}
  For $u^* \in D^*S|_X (\bar{x} \mid \bar{u})^{-1} (0)$,it is equivalent that
  $(0, -u^* ) \in N_{\gph S|_X}(\bar{x}, \bar{u}).$ As ${\rm proj}_{T_X(\bar{x}) } (0) =0$, we have $(0, -u^* ) \in {\rm proj}_{T_X(\bar{x}) \times \R^m } N_{\gph S|_X}(\bar{x}, \bar{u}).$ Then by definition of projectional coderivatives, $u^*  \in D^*_X S( \bar{x} \mid \bar{u})^{-1}(0)$.
\end{proof}

\begin{theorem}[generalized Mordukhovich criterion, {\cite[Theorem 2.4]{Meng2020}}]\label{thm-genMordcri}
 Consider a multifunction $S: \R^n\rightrightarrows  \R^m$ and $(\bar{x}, \bar{u})\in \gph S|_X$. Suppose that $\gph S$ is locally closed at $(\bar{x}, \bar{u})$ and that $X$ is closed and convex.  Then $S$ has the Lipschitz-like property relative to $X$ at $\bar{x}$ for $\bar{u}$ if and only if $D^*_{X}S(\bar{x}\mid\bar{u})(0)=\{0\}.$
 \end{theorem}

When $\bar{x}\in \inte X$, the projectional coderivative mapping,  $D^*_{X}S(\bar{x}\mid \bar{u})$ reduces to the coderivative $D^*S(\bar{x}\mid \bar{u})$ and accordingly, the generalized Mordukhovich criterion reduces to the Mordukhovich criterion, see \cite{mordukhovich1992sensitivity,VaAn}.

\section{Projectional coderivative and parametric systems}\label{sect:ParaSys}

In this section, we give a complete comparison between sufficient conditions for relative Lipschitz-like property in \cite{Benko2020,Meng2020} and derive some upper estimates of projectional coderivative for fully parametric system \cref{eqn:defofS-ParaSys}. 

First we introduce the definitions  of the directional limiting normal cone and the directional limiting coderivative. 
\begin{definition}[{\cite[Definition 2.3]{ginchevdirectionally}}]
  For a closed set  $\Omega\subset \R^n$ with $\bar{x}\in \Omega$  and a direction $u\in \R^n$, the directional limiting normal cone to $\Omega$ in direction $u$ at $\bar{x}$ is defined by
  \begin{equation}\label{eqn:defofDirNC}
      N_\Omega(\bar{x};u):=\limsup_{t\downarrow 0, \;u'\to u}\widehat{N}_\Omega(\bar{x}+tu'),
  \end{equation}
while for a set-valued mapping $S:\R^n\rightrightarrows  \R^m$ having locally closed graph around $(\bar{w}, \bar{x})\in \gph S$ and a pair of directions $(u, v)\in \R^n\times \R^m$, the set-valued mapping $D^*S((\bar{w},\bar{x});(u, v)):\R^m\rightrightarrows \R^n$, defined by, for all $v^*\in \R^m$,
\begin{equation}\label{eqn:defofDirCode}
    D^*S((\bar{w},\bar{x});(u, v))(v^*):=\{u^*\in \R^n\mid (u^*, -v^*)\in N_{\gph S}((\bar{w},\bar{x});(u, v))\},
\end{equation}
is called the directional limiting coderivative of $S$ in the direction $(u, v)$ at $(\bar{w}, \bar{x})$. 
\end{definition}

In the next theorem, we show that when we are referring to the set $X := \dom S$, the sufficient condition in \cite[Theorem 3.5]{Benko2020} implies the generalized Mordukhovich criterion when further $\dom S$ is closed and convex. For direct comparison, we adopt the explicit form introduced in \cite[Theorem 2.5]{Meng2020}. 

\begin{theorem}\label{thm:BenkoToGMC} 
Consider $S: \R^n\rightrightarrows  \R^m$, $\bar{x}\in X\subset \R^n$ and $\bar{u}\in S(\bar{x})$. Assume  that $\gph S$ is locally closed at $(\bar{x}, \bar{u})$ and that $\dom S$ is closed and convex around $\bar{x}$. Further assume that the following conditions are satisfied:
\begin{description}
  \item[(i)]  For every $x \in T_{\dom S} (\bar{x})$ and every sequence $t_k\downarrow 0$, there exists some $u\in \R^n$ such that
\[
\liminf_{k\to \infty}\frac{d((\bar{x}+t_kx,\bar{u}+t_ku), \gph S)}{t_k}=0.
\]
  \item[(ii)] The  equality
  \[
D^*S \left( (\bar{x},\bar{u}) ;(x, u) \right)(0) = \{0\}
\]
holds for all $x \in T_X(\bar{x})$ and $(x,u)\in T_{\gph S}(\bar{x}, \bar{u})$ with $(x, u)\not=(0, 0)$.
  \end{description}
Then $D^*_{\dom S}S(\bar{x}\mid \bar{u})(0) = \{0\} $ and $S$ has the Lipschitz-like property relative to $\dom S$ at $\bar{x}$ for $\bar{u}$. 
\end{theorem}
\begin{proof}
Let $t^* \in D^*_{\dom S} S (\bar{x}\mid \bar{u}) (0) $. By definition of projectional coderivative, there exist sequences 
$$(x_k, u_k) \xrightarrow[]{\gph S} (\bar{x},\bar{u}), \ (x^*_k, -u^*_k) \in N_{\gph S}(x_k, u_k), \ t^*_k = {\rm proj}_{T_{\dom S}(x_k)}(x^*_k)$$
 such that $t^*_k \rightarrow t^*$, $u^*_k \rightarrow 0$. Then by definition of normal cones, we know that there exist sequences 
 \begin{equation*}
     \small (x_{kt}, u_{kt}) \xrightarrow[]{\gph S}(x_k, u_k) , \ (x^*_{kt}, -u^*_{kt}) \in \hat{N}_{\gph S}(x_{kt} , u_{kt}) , \text{ such that } (x^*_{kt}, u^*_{kt}) \rightarrow (x^*_k, u^*_k).
 \end{equation*}
If $(x_k, u_k) \neq (\bar{x},\bar{u})$, let
\begin{equation}\label{construction} \small
    \tau_{kt} := \| (x_{kt} - \bar{x}, u_{kt} - \bar{u}) \|, \ (x'_{kt} , u'_{kt}) := \frac{(x_{kt} - \bar{x}, u_{kt} - \bar{u}) }{\|(x_{kt} - \bar{x}, u_{kt} - \bar{u})\| } , \ (x'_k , u'_k) := \frac{(x_{k} - \bar{x}, u_{k} - \bar{u}) }{\|(x_{k} - \bar{x}, u_{k} - \bar{u})\| }.
\end{equation}
Then we have $$(x'_{kt}, u'_{kt}) \rightarrow  (x'_k, u'_k) ,\ \tau_{kt} \searrow 0 \text{ and } (x^*_{kt}, -u^*_{kt}) \in \hat{N}_{\gph S} \left( (\bar{x}, \bar{u}) + \tau_{kt}(x'_{kt}, u'_{kt}) \right).$$ By definition of directional normal cone we have $(x^*_k, -u^*_k) \in N_{\gph S}\left( (\bar{x},\bar{u}); (x'_k, u'_k) \right).$
Besides, by definition of tangent cones, $(x'_k, u'_k) \in T_{\gph S}(\bar{x}, \bar{u}) \cap \Sph$ (where $\Sph$ denotes the unit sphere) and $x'_k \in T_{\dom S}(\bar{q})$. By condition (ii),  we have that $u^*_k \rightarrow 0 $ indicates $x^*_k  \rightarrow 0$. Therefore $$t^*_k  = {\rm proj}_{T_{dom S}(x_k)} (x^*_k) \rightarrow 0 = t^*.$$

If $(x_k, u_k) = (\bar{x},\bar{u})$, then we have $$(x^*, -u^*) \in N_{\gph S}(\bar{x}, \bar{u}) \text{ and }t^*  =  {\rm proj}_{T_{\dom S}(\bar{x})}(x^*), \ u^* =0.$$ Without loss of generality, we can further assume that $(x^*, -u^*) \in \hat{N}_{\gph S}(\bar{x},\bar{u})$, as otherwise by the construction in \eqref{construction} we can always find some $(x'_k, u'_k) \in T_{\gph S}(\bar{x},\bar{u}) \cap\Sph$ and the argument is similar to the above.
Given $t^* \in T_{\dom S}(\bar{x})$, by condition (i),  we can find accordingly $u'$ such that $(t^*, u') \in T_{\gph S}(\bar{x},\bar{u})$. By polar relation between tangent cones and regular normal cones, we have 
$$\langle (t^*, u'), (x^*, -u^*) \rangle  = \langle t^*, x^* \rangle \leq 0.$$
Given the convexity of $T_{\dom S}(\bar{x})$, the decomposition of $x^* = t^* + y^*$ is unique with $t^* = {\rm proj}_{T_{\dom S}(\bar{x})}(x^*)$ and $y^* = {\rm proj}_{N_{\dom S}(\bar{x})}(x^*)$, and $t^* \perp y^*$ by \cite[Exercise 12.22]{VaAn}. Therefore we have $t^* =0$ and $D^*_{\dom S} S(\bar{x}\mid \bar{u})(0) = \{0\}$. By \cref{thm-genMordcri} $S$ is Lipschitz-like relative to $\dom S$ at $\bar{x}$ for $\bar{u}$.
\end{proof}

Next we give an example where a fixed-point expression of the projectional coderi- vative can be given when the set $X$ is a smooth manifold around the point $\bar{x}$, (see \cite[Example 6.8]{VaAn} for the definition of smooth manifold).
\begin{lemma}[Projectional coderivatives of a set-valued mapping restricted on a smooth manifold]\label{Prop:SM-ProjCode}
Consider $S: \R^n\rightrightarrows  \R^m$ and  $\bar{u}\in S(\bar{x})$.   Suppose that $\gph S$ is locally closed at
 $(\bar{x}, \bar{u})$ and $X$ is a smooth manifold around $\bar{x}$. Then we have 
 \begin{equation*}
     D^*_X S(\bar{x}\mid\bar{u})(u^*)={\rm proj}_{T_X(\bar{x})}D^*S|_X(\bar{x}\mid\bar{u})(u^*), \quad \forall u^*.
 \end{equation*}
\end{lemma}
\begin{proof} 
By \cref{def:ProjCode}, it suffices to show
   \begin{equation*}
       D^*_X S(\bar{x}\mid\bar{u})(u^*)\subseteq {\rm proj}_{T_X(\bar{x})}D^*S|_X(\bar{x}\mid\bar{u})(u^*), \quad \forall u^*.
   \end{equation*}
  Let $y^*\in D^*_X S(\bar{x}\mid\bar{u})(u^*)$. Then there exist some  sequences $(x_k,u_k)\xrightarrow[]{\gph S|_X}(\bar{x},\bar{u})$ and
  $x_k^*\in D^*S|_X(x_k\mid u_k)(u_k^*)$ such that $u_k^*\to u^*$ and $y_k^*:={\rm proj}_{T_X(x_k)}(x_k^*)\to y^*$. By representation of tangent cone of smooth manifold \cite[Example 6.8]{VaAn}, $T_X(\cdot)$ is continuous at $\bar{x}$ relative to $X$. Together with \cite[Exercise 5.35]{VaAn}, we have that $y^* = {\rm proj}_{T_X(\bar{x})}(x^*) $. This completes the proof.
\end{proof}

We now investigate upper estimates of projectional coderivative of the mapping $S$ defined in \cref{eqn:defofS-ParaSys}. Recall that upper estimates of coderivatives of $S$ have been given in \cite[Theorem 2.1]{LevyMord2004}. We also discuss the upper estimates under two special cases: (i) $G(w,x) =0$, (ii) $M(w,x) = M(x)$ and $X = \dom S$.
\begin{theorem}\label{Thm-ParametricSystem}
Consider the implicit mapping $S: \R^m \rightrightarrows \R^n$ of the form \cref{eqn:defofS-ParaSys} 
with $G: \R^{m+n} \rightarrow \R^d$ a $\mathcal{C}^1 $ mapping, and $M: \R^{m + n} \rightrightarrows \R^d$ a multifunction with  closed graph. Consider a pair $(\bar{w},\bar{x}) \in \gph S|_{W}$ where $W \subseteq \dom S$ is a closed set. Let $\mathcal{M}(w,x,y) :=  \nabla G(w,x) ^*y +D^*M|_{W\times \R^n}  \left( \left(w,x \right) \left| -G(w,x)\right.\right)(y) $. If the following basic constraint qualification holds: 
\begin{equation}\label{cq:para-sum-basic}
           (0,0)\in  \mathcal{M}(\bar{w},\bar{x},y)   \Longrightarrow y=0,
  \end{equation}
then we have
\begin{equation}\label{incl:ParaSys-sum-1}
  \begin{split}
  D^*_{W} S\left( \bar{w}  \mid \bar{x} \right) (r) \subseteq & \limsup_{\substack{ (w,x)\xrightarrow[]{\gph S|_W} (\bar{w},\bar{x}) \\ {  r'\rightarrow r }}} \bigcup_{y \in \R^d} \bigg\{  {\rm proj}_{T_{W}(w)} (v) \ \bigg| \   (v,-r') \in \mathcal{M}(w,x,y)   \bigg\}.
  \end{split}
\end{equation}
If the following strong constraint qualification is satisfied:
  \begin{equation}\label{cq:para-sum-senior}
     (0, 0)\in  \limsup_{\substack{ (w,x)\xrightarrow[]{\gph S|_W} (\bar{w},\bar{x}) \\ y'\rightarrow y}} {\rm proj}_{T_W(w)\times \R^n }   \mathcal{M}(w,x,y')   \Longrightarrow y=0,
  \end{equation}
then the limsup in \cref{incl:ParaSys-sum-1} can be equivalently put into the bracket as 
\begin{equation}\label{incl:ParaSys-sum-2}
  \begin{aligned}
  D^*_{W} S& \left( \bar{w}  \mid \bar{x} \right) (r) \subseteq \bigg\{ t \in \R^m \bigg|  \exists  y \in \R^d \text{ with } (t,-r) \in    \\
       &  \hskip1cm   \limsup_{ \substack{\tiny  (w ,x,-G(w,x))\xrightarrow[]{\gph M|_{W\times \R^n} }(\bar{w}, \bar{x},-G(\bar{w}, \bar{x})) \\ y' \rightarrow y} }  {\rm proj}_{T_{W}(w)\times \R^n}  \mathcal{M}(w,x,y')   \bigg\}.
  \end{aligned}
\end{equation}
If in addition,  $M|_{W\times \R^n}$ is graphically regular at $(\bar{w}, \bar{x}, -G(\bar{w},\bar{x}))$ and $W$ is a smooth manifold around $\bar{w}$, then
\begin{equation}\label{eqn:ParaSys-sum-eqn}
   D^*_{W} S\left( \bar{w}  \mid \bar{x} \right)(r) =   \bigg\{ t \in \R^m  \bigg|  \exists  y \in \R^d \text{ with } (t,-r) \in   {\rm proj}_{T_{W}(\bar{w})\times \R^n} \mathcal{M}(\bar{w},\bar{x}, y) \bigg\}.
\end{equation}
\end{theorem}
\begin{proof}
By a simple statement of contradiction with the osc of $D^*M|_{W\times \R^n}$, \cref{cq:para-sum-basic} also indicates that
\begin{equation*}
  (0,0)\in  \mathcal{M}(w,x, y)  \Longrightarrow y=0
\end{equation*}
for any $(w,x) \in \gph S|_W$ sufficiently close to $(\bar{w}, \bar{x})$. According to \cite[Theorem 2.1]{LevyMord2004}, for any pair $(w,x) \in \gph S|_W$ sufficiently near $(\bar{w},\bar{x})$,  we have
\begin{equation}\label{incl:para-sum-nc1}
  N_{\gph S|_W} (w,x) \subseteq  \bigcup_{y\in \R^d} \mathcal{M}(w,x, y) .
\end{equation}
Therefore we have
\begin{align*}
     &  \limsup_{\tiny  (w ,x)\xrightarrow[]{\gph S|_{W} }(\bar{w}, \bar{x})} {\rm proj}_{T_W(w) \times \R^n } N_{\gph S|_W} (w,x) \nonumber\\
     & \hskip3cm \subseteq \limsup_{\tiny  (w ,x)\xrightarrow[]{\gph S|_{W} }(\bar{w}, \bar{x}) }  \bigcup_{y\in \R^d}  {\rm proj}_{T_W(w) \times \R^n } \mathcal{M}(w,x, y) 
\end{align*}
and accordingly the inclusion \cref{incl:ParaSys-sum-1} holds.

Now, assume that the strong constraint qualification \cref{cq:para-sum-senior} holds. 
By nature of projection and the outer limit, we can also see that \cref{cq:para-sum-senior} indicates \cref{cq:para-sum-basic}.
For $t$ belonging to the right-hand side of (\ref{incl:ParaSys-sum-1}), there exist sequences $(w_k ,x_k) \xrightarrow[]{\gph S|_W}(\bar{w},\bar{x})$, $y_k\in \R^d$ and $(v_k ,-r_k) \in \mathcal{M}(w_k,x_k, y_k)$,
such that $t_k \in {\rm proj}_{T_W(w_k)} (v_k) \rightarrow t$ and $r_k \rightarrow r$.
Taking a subsequence if necessary, we have either $y_k\rightarrow y\in \R^d$ or $\lambda_k y_k \rightarrow y\in \R^d$ with $\lambda_k \searrow 0$.
For the first case, we directly have that $t$ belongs to the right-hand side of \cref{incl:ParaSys-sum-2}.
For the second case, without loss of generality we assume $\|y \| =1$. With the conic structure we have
$ \lambda_k (v_k ,-r_k) \in \mathcal{M}(w_k,x_k, \lambda_k y_k) $ and accordingly $ \lambda_k t_k \in \lambda_k {\rm proj}_{T_W(w_k)} (v_k)  \rightarrow 0$,  $\lambda_k r_k \rightarrow 0$, which contradicts the constraint qualification \cref{cq:para-sum-senior} with $\|y \| =1$. Thus the second case is not possible. 
Given that $(w,x) \xrightarrow[]{\gph S|_W}(\bar{w}, \bar{x})$ is equivalent to $(w,x, -G(w,x)) \xrightarrow[]{\gph M|_{W\times \R^n}}(\bar{w}, \bar{x}, -G(\bar{w}, \bar{x}))$ and therefore we have that $t$ also belongs to the set on the right-hand side of \cref{incl:ParaSys-sum-2}. Note that in general the right-hand side of \cref{incl:ParaSys-sum-2} is included by that of \cref{incl:ParaSys-sum-1} and therefore, with constraint qualification \cref{cq:para-sum-senior} being satisfied, these two sets are identical. 

If furthermore $M|_{W\times \R^n}$ is graphically regular at $(\bar{w}, \bar{x}, -G(\bar{w},\bar{x}))$, again by \cite[Theorem 2.1]{LevyMord2004},  we have \cref{incl:para-sum-nc1} as an equation at the reference point $(\bar{w}, \bar{x})$ and therefore
\begin{equation}\label{ProjCode-Para-LowerBound}
\begin{aligned}
       &  \ {\rm proj}_{T_W(\bar{w})} D^*S|_W(\bar{w}\mid \bar{x}) (r)  \\
  =  & \ \bigg\{ t \in \R^m \bigg|  \exists  y \in \R^d \text{ with } (t,-r) \in  {\rm proj}_{T_{W}(\bar{w})\times \R^n} \mathcal{M}(\bar{w},\bar{x}, y) \big) \bigg\} \\
  \subseteq &\ D^*_W S(\bar{w} \mid \bar{x})(r) .
\end{aligned}
\end{equation}
Besides, when $W$ is a smooth manifold at $\bar{w}$, by \cref{Prop:SM-ProjCode} and \cref{incl:para-sum-nc1},
\begin{equation}\label{ProjCode-Para-UpperBound}
    \begin{aligned}
  & \ D^*_{W} S\left( \bar{w}  \mid \bar{x} \right) (r) = {\rm proj}_{T_W(\bar{w})} D^*S|_W \left(\bar{w} \mid \bar{x}\right) (r)  \\
  & \hskip1cm \subseteq \bigg\{ t \in \R^m \bigg|  \exists  y \in \R^d \text{ with } (t,-r) \in   {\rm proj}_{T_{W}(\bar{w})\times \R^n}  \mathcal{M}(\bar{w},\bar{x}, y)\bigg\}.
 \end{aligned}
\end{equation}

Combining the conditions that $M|_{W\times \R^n} $ is graphically regular at $(\bar{w}, \bar{x}, -G(\bar{w},\bar{x}))$ and that $W$ is a smooth manifold at $\bar{w}$,  \cref{ProjCode-Para-LowerBound} and \cref{ProjCode-Para-UpperBound} turn into equation \cref{eqn:ParaSys-sum-eqn}.
\end{proof}

Next we use a simple example to illustrate how the strong constraint qualification (\ref{cq:para-sum-senior}) can be applied in calculating the projectional coderivative \cref{eqn:ParaSys-sum-eqn}.
\begin{example}
For $S(w): =\left\{ x\in \R^n \mid   Ax + w\in K \right\} $ where $K \subseteq \R^m$ is a closed set. Let $G(w,x) = -Ax-w$ and $M(w,x) = K$.
For $W\subseteq \dom S$ we can write $\gph M|_{W\times \R^n} = W \times \R^n \times K$ and accordingly 
\begin{eqnarray*}
    D^*M|_{W\times \R^n} \left( (w,x)\mid u\right) (y) &=&\begin{cases}
                 N_{W}(w) \times \{0\} , & \text{ if } y\in -N_K(u),\\
                 \emptyset, & \text{ if } y\notin -N_K(u),
                \end{cases}\\
    \nabla G(w,x) ^* y &=& (-y, -A^*y).
\end{eqnarray*}
Let $n= m =2$, $K = \R\times \{0\} \cup \{0\}\times \R$, $A = \left(\begin{smallmatrix}0&0 \\0&1 \end{smallmatrix} \right)$. Then $\dom S = K + \mbox{range} A = \R^2$. Consider the particular pair $(\bar{w}, \bar{x} ) \in \gph S|_W$ where $\bar{w} = (0,1)^\top $, $\bar{x} = (0,0)^\top$ and a smooth manifold $W = \R\times \{1\} \subseteq \dom S$.
Then the constraint qualification \cref{cq:para-sum-senior}, by \cref{Prop:SM-ProjCode}, becomes
\begin{align*}
    (0, 0) \in & \ {\rm proj}_{T_W(\bar{w}) \times \R^n}\mathcal{M}(\bar{w},\bar{x}, y) \\
    = & \left\{ ({\rm \proj}_{T_W(\bar{w} )} (v-y), -A^*y ) \mid v\in N_W(\bar{w}) , y\in -N_K(A\bar{x} + \bar{w})\right\}  \Longrightarrow y =0, 
\end{align*}
As $-G(\bar{w},\bar{x}) =A\bar{x} + \bar{w} =  (0,1)^\top$, $K$ is regular at $-G(\bar{w}, \bar{x})$ and $N_K(-G(\bar{w}, \bar{x}) ) = N_K\left( (0,1)^\top \right) = \R\times \{0\}.$ Thus $M|_{W\times \R^n}$ is graphically regular at $\left( \bar{w}, \bar{x},-G(\bar{w}, \bar{x}) \right) $. 
Besides, in view of the facts that  $T_W(\bar{w}) = \R\times \{0\}$ and $y\in -N_K(A\bar{x}+ \bar{w}) = \R\times \{0\}$ and by the polar relation between $N_W(\bar{w})$ and $T_W(\bar{w})$,
$$0 = {\rm proj}_{T_W(\bar{w})}(v-y) = {\rm proj}_{T_W(\bar{w})} (-y) = -y \Longrightarrow y =0.$$ 
Thus the constraint qualification \cref{cq:para-sum-senior} is satisfied. 
Applying \cref{eqn:ParaSys-sum-eqn}, 
we obtain 
$$    D^*_W S(\bar{w}, \bar{x}) (r) =  \left\{ y \mid  y \in N_K(A\bar{w}+ \bar{x}) \text{ with } A^* y= r \right\} 
     =  \begin{cases}
        \R \times \{0\} & \text{ if } r =(0,0)^\top, \\ 
        \emptyset & \text{ if } r\neq (0,0)^\top.
     \end{cases}
$$
Thus, $D^*_W S(\bar{w}, \bar{x}) ((0,0)^\top) =   \R \times \{0\} \neq \{(0,0)^\top\}$. As $W$ is also a convex set, $S$ does not enjoy the Lipschitz-like property relative to $W$ at $\bar{w}$ for $\bar{x}$  according to the generalized Mordukhovich criterion (\cref{thm-genMordcri}).
\end{example}

By observing the right-hand side of the expression \cref{incl:ParaSys-sum-2}, we can see that $(t,-r)$ actually belongs to something that is very close to the projectional coderivative of the multifunction $G(w,x) + M(w,x)$ relative to the set $W\times \R^n$ at $(\bar{w}, \bar{x})$ for $0$.  Next we present a simpler model by taking $G(w,x) = 0$ so that the relation of projectional coderivatives between $S$ and $M$ can be revealed more clearly. 

\begin{corollary}\label{coro:ImpSys}
  For an implicit mapping $S : \R^m \rightrightarrows \R^n$ as
  $$ S(w) = \left\{ x \mid 0\in M(w,x) \right\}$$ where $M: \R^m \times \R^n \rightrightarrows \R^d$ has closed graph. Let a closed set $W\subseteq \R^m$ and let $\bar{x} \in S|_W (\bar{w})$. If the basic constraint qualification holds
  \begin{equation}\label{cq:para-sing-basic}
       (0,0) \in D^* M|_{W\times \R^n} \left( \left(\bar{w},\bar{x} \right)\mid 0 \right) (y) \Longrightarrow y=0,
  \end{equation}
  then 
  \begin{equation}\label{incl:ParaSys-sing-1}
  \begin{aligned}
     & D^*_{W} S\left( \bar{w}  \mid \bar{x} \right) (r) \subseteq \\
     & \limsup_{ \substack{ (w,x)\xrightarrow[]{\gph S|_W} (\bar{w},\bar{x}) \\ r'\rightarrow r}} \bigcup_{ y \in \R^d } \bigg\{  {\rm proj}_{T_{W}(w)} (v) \ \bigg|  (v,-r') \in   D^* M|_{W\times \R^n} \left( (w , x )\mid 0 \right) (y)  \bigg\}.
     \end{aligned}
\end{equation}
If the strong constraint qualification holds
\begin{equation}\label{cq:para-sing-senior}
   (0,0) \in D^*_{W\times \R^n} M\left( \left(\bar{w},\bar{x} \right)\mid 0 \right) (y) \Longrightarrow y=0
\end{equation}
  then we have
  \begin{equation}\label{incl:ParaSys-sing-2}
    D^*_W S(\bar{w}\mid \bar{x}) (r) \subseteq \left\{ \left. t \ \right| \  \exists y \text{ s.t. } (t, -r)\in D^*_{W\times \R^n}M ((\bar{w}, \bar{x}) \mid 0) (y)\right\}.
  \end{equation}
  When in addition $M|_{W\times \R^n}$ is graphically regular at $(\bar{w}, \bar{x},0)$ and $W$ is a smooth manifold at $\bar{w}$, the inclusions \cref{incl:ParaSys-sing-1} and \cref{incl:ParaSys-sing-2} are identical and become equations. 
\end{corollary}
\begin{proof}
   This corollary comes from direct application of \cref{Thm-ParametricSystem} by taking $G(w,x) = 0 $. As $ (w,x)\xrightarrow[]{\gph S|_W} (\bar{w},\bar{x})$ is equivalent to $(w,x,0)\xrightarrow[]{\gph M|_{W\times \R^n}} (\bar{w},\bar{x},0)$, therefore
  $$ \limsup_{ \substack{ (w,x)\xrightarrow[]{\gph S|_W} (\bar{w},\bar{x}) \\ y'\rightarrow y}} \hskip -0.1cm  {\rm proj}_{T_W(w)\times \R^n }   D^*M|_{W\times \R^n}  \left( \left(w,x \right) \hskip -0.1cm \mid \hskip -0.1cm 0 \right) (y')  \subseteq D^*_{W\times \R^n} M\left( (\bar{w}, \bar{x}) \hskip -0.1cm \mid \hskip -0.1cm 0 \right) (y).$$
   Then we can rewrite the inclusions in \cref{Thm-ParametricSystem} as  \cref{incl:ParaSys-sing-1} and \cref{incl:ParaSys-sing-2} respectively. 
   \end{proof}
   

In \cref{Thm-ParametricSystem}, two different constraint qualifications are mentioned. We can see that the basic one \cref{cq:para-sum-basic} is ensuring the upper estimate in two ways: (i) restricting $S$ to $W$; (ii) expressing the normal cone of $\gph S$ via those of $\gph G$ and $\gph M$. The strong one, as shown in \cref{cq:para-sing-senior}, aims at presenting the projectional coderivative of $S$ via that of $M$ when $G$ vanishes.  In the next theorem, we give a setting where the basic constraint qualification \cref{cq:para-sum-basic} is automatically satisfied when we consider the largest possible $W$ as $\dom S$. 
\begin{theorem}\label{thm:ParaSys-ProjCde-domS}
  For $S$ defined in \cref{eqn:defofS-ParaSys}, if $M(w,x) = M(x)$ and $\nabla_w G(w,x)$ has full rank around $(\bar{w},\bar{x})\in \gph S$, then 
  \begin{equation}\label{eqn:ParaSys-domS-ProjCde}
  \begin{split}
     D^*_{\dom S} S\left( \bar{w}  \mid \bar{x} \right) (r) = & \limsup_{\substack{ (w,x)\xrightarrow[]{\gph S} (\bar{w},\bar{x}) \\ {  r'\rightarrow r }}}  \bigcup_{y \in \R^d} \bigg\{ {\rm proj}_{T_{dom S}(w)} (\nabla_w G(w,x) ^*y) \\ & \bigg| -r'\in \nabla_x G(w,x)^* y + D^* M\left(  x \mid -G(w,x) \right) (y)   \bigg\}.
  \end{split}
\end{equation}
\end{theorem}
\begin{proof}
 When the set $W:= \dom S$, $S|_W = S$. By condition (b) in \cite[Theorem 2.1]{LevyMord2004}, we have for any $(w,x) \in \gph S$,
 $$N_{\gph S}(w,x) = \bigcup_{y\in \R^d} \left(  \nabla G(w,x)^* y + \{0\} \times D^*M\left(x\mid -G(w,x) \right)(y) \right) . $$
Therefore we have
\begin{align}
     &  \limsup_{\tiny  (w ,x)\xrightarrow[]{\gph S|_{W} }(\bar{w}, \bar{x})} {\rm proj}_{T_W(w) \times \R^n } N_{\gph S|_W} (w,x)  \nonumber\\
     = &  \limsup_{\tiny  (w ,x)\xrightarrow[]{\gph S }(\bar{w}, \bar{x})} {\rm proj}_{T_{\dom S}(w) \times \R^n } N_{\gph S} (w,x) \nonumber\\
     =  & \limsup_{\tiny  (w ,x)\xrightarrow[]{\gph S }(\bar{w}, \bar{x}) }  \bigcup_{y\in \R^d}  {\rm proj}_{T_{\dom S}(w) \times \R^n } \left(  \nabla G(w,x)^* y +  \left\{0 \right\} \times D^*M (x\mid -G(w,x) )(y) \right) \nonumber
\end{align}
and thus the equality \cref{eqn:ParaSys-domS-ProjCde}. 
\end{proof}


\section{Affine variational inequalities}\label{sect:AVI}
In this section, we consider the following affine variational inequality (AVI):
$$0 \in q+Mx + N_C(x)$$ where $C \subset \R^n$ is a polyhedral set and $M$ is an $n\times n$ matrix. Here we consider that $q$ is a parameter. The solution mapping of AVI is written as:
\begin{equation}\label{eqn:AVI}
  S(q) = \left\{x  \mid  0 \in q+Mx + N_C(x) \right\}.
\end{equation}

Note that for AVI, $\gph S$ is always a union of finitely many polyhedral sets as it is a linear transformation of $\gph N_C$ (see \cite[Example 12.31]{VaAn} and \cite{dontchev1996}). For a closed subset $Q \subseteq \dom S$, the graph of the multifunction $S$ restricted on $Q$ is
\begin{equation}\label{eqn:AVI-gphofLQ}
  \gph S|_Q = \gph S \cap \left( Q \times \R^n \right) = \left\{ (q,x) \in Q \times \R^n \mid 0 \in q+Mx + N_C(x) \right\}.
\end{equation}
We first obtain a upper estimate of the projectional coderivative of $S$, when $Q$ is a union of polyhedral sets. In this way, we can skip the limsup appearing in e.g. (\ref{incl:ParaSys-sum-1}) and express the upper estimate of $D^*_{Q}S$ in the form of a union within the range of a ball centered at $(\bar{q},\bar{x})$. 
\begin{proposition}\label{prop:AVI-ProjCde-Expression}
  For the solution mapping $S$ \cref{eqn:AVI} of AVI, consider a union of polyhedral sets $Q\subseteq \dom S$. Let $(\bar{q}, \bar{x}) \in \gph S|_Q$. If the following basic constraint qualification holds: 
     \begin{equation}\label{cq:AVI-ProjCde}
     u^* \in N_Q(\bar{q}), \  (M^* u^*, u^*) \in N_{\gph N_C}\left(\bar{x}, -M\bar{x}-\bar{q} \right) \Longrightarrow u^* =0,
   \end{equation}
   then 
   \begin{equation}\label{eqn:AVI-ProjCode}
          \begin{aligned}
     D^*_{Q} S( \bar{q} \mid \bar{x}) (y^* ) \subseteq
             &  \bigcup_{(q,x) \in \gph S|_Q \cap \B_{\varepsilon}(\bar{q}, \bar{x})} \bigg\{ \left.{\rm proj}_{T_Q(q)} (-u^* + w^*) \ \right|  \ w^*\in N_Q(q) ,  \\[-0.1cm]
     & \hskip1cm  \exists \  u^* \text{ s.t. } (M^*u^* -y^*, u^*)\in N_{\gph N_C}(x, -Mx-q) \bigg\}
   \end{aligned}
   \end{equation}
  for sufficiently small $\varepsilon >0$. If one of the following conditions is satisfied: 
  \begin{enumerate}[label=\normalfont(\alph*)]
      \item $\gph N_C$ and $Q$ are regular around $(\bar{x}, -M\bar{x} - \bar{q})$ and $\bar{q}$ respectively, 
      \item $Q = \dom S$, (in this case the constraint qualification \cref{cq:AVI-ProjCde} can be avoided),
  \end{enumerate}
then the inclusion \cref{eqn:AVI-ProjCode} becomes an equality. 
\end{proposition}
\begin{proof}
For any $q, x \in \R^n$, let 
\[\Gamma|_{Q\times \R^n}(q,x):= N_C(x),   \ \mbox{if } q \in Q; \ \ \emptyset,  \ \mbox{if } q\notin Q.\]
For any $q\in Q, x \in \R^n \mbox{ and } v \in N_C(x)$, we have
$$D^*\Gamma|_{Q\times \R^n}\left( (q, x) \mid v\right) = N_Q(q) \times D^*{N_C}(x\mid v).$$ 
Then the mapping $S$ \cref{eqn:AVI} is rewritten as
\begin{equation} \label{aviS2}
S|_Q(q) = \{ x \mid 0 \in q + Mx + \Gamma|_{Q\times \R^n}(q,x)\}.
\end{equation}

Note that the constraint qualification (\ref{cq:para-sum-basic}) becomes
\begin{equation*}
    \begin{aligned}
    (0,0) = (u^*, M^* u^*)  + (w^*,v^*) \text{ with } w^* \in N_Q(\bar{q}), \ & v^* \in D^*N_C(\bar{x}\mid  -M\bar{x} - \bar{q})(u^*) \\
    & \Longrightarrow u^* =0,
    \end{aligned}
\end{equation*}
which is equivalent to 
$$-u^* \in N_Q(\bar{q}),\  -M^* u^* \in D^*N_C(\bar{x} \mid -M\bar{x} - \bar{q}) (u^*) \Longrightarrow u^* =0. $$
By tuning the direction of $u^*$, we arrive at (\ref{cq:AVI-ProjCde}). 

And the upper estimate (\ref{incl:ParaSys-sum-1}) can be put as 
\begin{equation*}
    \begin{aligned}
     D^*_{Q} S\left( \bar{q}  \mid \bar{x} \right) (y^*) \subseteq &  \limsup_{\substack{ (q,x)\xrightarrow[]{\gph S|_Q} (\bar{q},\bar{x}) \\  y'^*\rightarrow y^* }} \bigcup_{u^* \in \R^n} \bigg\{ \left. {\rm proj}_{T_Q(q)} (t^*) \ \right| \ ( t^*, -y'^*) \in (u^*,  M^* u^*) +  \\
     & \hskip4cm  \   N_Q(q) \times D^*N_C(x \mid -Mx-q) (u^*) \bigg\} \\
     = &  \limsup_{\substack{ (q,x)\xrightarrow[]{\gph S|_Q} (\bar{q},\bar{x}) \\  y'^*\rightarrow y^* }} \bigcup_{u^* \in \R^n} \bigg\{ \left. {\rm proj}_{T_Q(q)} (u^* + w^*) \ \right| \ y'^* = - M^* u^* - v^*, \\ 
     &   \hskip3cm w^*\in N_Q(q), v^*\in D^*N_C(x \mid -Mx-q) (u^*) \bigg\} \\
     = &  \bigcup_{(q,x) \in \gph S|_Q \cap \B_{\varepsilon}(\bar{q}, \bar{x})}  \bigcup_{u^* \in \R^n}\bigg\{ \left.{\rm proj}_{T_Q(q)} (u^* + w^*)  \right|  y^* = - M^* u^* - v^*, \\
     & \hskip3cm w^*\in N_Q(q), v^*\in D^*N_C(x \mid -Mx-q) (u^*) \bigg\}
\end{aligned}
\end{equation*}
for sufficiently small $\varepsilon >0$. Here the second equality comes from the polyhedrality of both $\gph S|_Q$ and $Q$ as only finitely many possible statuses are considered.  By adjusting the expression of $y^*$ and direction of $u^*$, we finally arrive at \cref{eqn:AVI-ProjCode}. Under condition (a), we have $N_{\gph S|_Q}(q,x)  = N_{Q}(q) \times \{0\} + N_{\gph S}(q,x)$ for $(q,x)$ near $(\bar{q}, \bar{x})$ by \cite[Theorem 6.14]{VaAn}. Accordingly the inclusion in \cref{eqn:AVI-ProjCode} becomes an equality. Under condition (b), the equality in \cref{eqn:AVI-ProjCode} is derived as an application of \cref{thm:ParaSys-ProjCde-domS}.
\end{proof}

Given the upper estimate \cref{eqn:AVI-ProjCode}, we can give a sufficient condition for the Lipschitz-like property of $S$ relative to $Q$ when $Q$ is further convex, that is $Q$ is a polyhedral set. Moreover, based on the critical face condition introduced in \cite{dontchev1996}, the sufficient condition can be simplified concerning the given point only. As such we derive a `generalized critical face condition'. To do this, let us review two notations introduced in \cite{dontchev1996}. For a polyhedral set $C \subset \R^n$ and $(x,v) \in \gph N_C$, the critical cone $K(x,v)$ is defined by
$$K(x,v)= T_C (x) \cap [v]^\perp.$$
Let $\mathcal{F} (K)$ be the collection of all the closed faces of polyhedral cone $K$ in the form of $F = K \cap [v^*] ^\perp, \text{ with }  v^* \in K^*$. It is clear that such $F$ is also a polyhedral cone.

The following lemma plays an important role in the development of critical face condition in \cite{dontchev1996}, which will be useful here as well.

\begin{lemma}[{\cite[Reduction Lemma]{dontchev1996}}] \label{lm:ReductionLemma}
  For any $(x,v) \in \gph N_C$, there is a neighborhood $U$ of $(0,0)$ in $ \R^n \times \R^n$ such that, for $(x',v') \in U$, one has
  \begin{equation*}
     v + v' \in N_C (x + x') \Longleftrightarrow v' \in N_{K(x,v)} (x').
  \end{equation*}
  In particular, $T_{\gph N_C} (x,v) = \gph N_{K(x,v)}$.
\end{lemma}

With Reduction Lemma, the local geometry of $\gph N_C$ around $(\bar{x}, \bar{v})$ can be observed via that of $\gph N_{K(\bar{x}, \bar{v})}$ and allows us to express $N_{\gph N_C}(x,v)$ via the faces of the critical cone $K(x,v)$. 
From the proof of \cite[Theorem 2]{dontchev1996}, for any pair $(x, v)  \in \gph N_C$, we have
\begin{equation}\label{eqn:AVI-faceexpGphNc}
  N_{\gph N_C}(x, v) = \left\{ \left( F_1 -F_2 \right)^*  \times  \left( F_1 -F_2 \right) \mid  F_2 \subset  F_1 \in \mathcal{F} (K(x, v)) \right\}.
\end{equation}

In \cite{GfrOut2016}, the directional limiting normal cone of $\gph N_C$ is expressed with critical faces. In light of this expression, we have the following result. 
\begin{lemma} \label{lm:EquivforDirNCandNC}
For a given pair $(x,v) \in \gph N_C$ and  $(x',v') \in T_{\gph N_C}(x,v)$ sufficiently near to $(0,0)$, we have
\begin{equation}\label{NCexpNBH}
    \begin{aligned}
        N_{\gph N_C} (x+x',v+v') =  \big\{ \left( F_1 -F_2 \right)^*  \times  \left( F_1 -F_2 \right) \mid  & \ x'\in F_2 \subset  F_1 \subset [v']^\perp,\\
        & \quad F_1, F_2 \in \mathcal{F} (K(x, v)) \big\} .
    \end{aligned}
\end{equation}
\end{lemma}
\begin{proof}
As $(x',v') \in T_{\gph N_C}(x,v)$ is sufficiently near to $(0,0)$, we have
  \begin{equation*}
    \begin{aligned}
    N_{\gph N_C} (x+x',v+v') = & \limsup_{(x'', v'')\xrightarrow[]{} (x',v') } \widehat{N}_{\gph N_C} \left((x,v) + (x'',v'') \right)  \\[-0.1cm] 
     =  & \limsup_{\substack{ t\searrow 0 \\ (\tilde{x}, \tilde{v}) \rightarrow(x',v')   }} \widehat{N}_{\gph N_C} \left((x,v) + t(\tilde{x}, \tilde{v}) \right) \\
     = & \ N_{\gph N_C}\left( (x, v); ( x',v')\right).
\end{aligned}
\end{equation*}
The last equality comes from the definition of directional limiting normal cone \cref{eqn:defofDirNC}. Then \cref{NCexpNBH} is obtained by \cite[Theorem 2.12]{GfrOut2016}.
\end{proof}

In the following theorem, we present the `generalized critical face condition' as an inclusion, which reduces to the one in \cite{dontchev1996} when $\bar{q} \in \inte  Q$.
\begin{theorem}\label{thm:genCritFace}
  For $(\bar{q}, \bar{x})\in \gph S|_Q$ and $Q$ being a polyhedral set, suppose the following constraint qualification holds:
    \begin{equation}\label{eqn:AVI-CQface}
   N_Q(\bar{q})\cap (F_1 -F_2 ) \cap (M(F_1-F_2))^*=\{0\},
  \end{equation}
where $F_1, F_2 \in \mathcal{F}(K(\bar{x}, \bar{v}))$ are closed faces with $F_2 \subset F_1$ and $\bar{v}= -M\bar{x} -\bar{q}$. If, for any such $F_1, F_2$, the following inclusion is satisfied:
    \begin{equation}\label{eqn:AVI-SufCondface}
     (F_1 -F_2 ) \cap (M(F_1-F_2))^* \subseteq -N_Q(\bar{q}),
  \end{equation}
then $S$ has Lipschitz-like property relative to $Q $ at $\bar{q}$ for $\bar{x}$. The condition \cref{eqn:AVI-SufCondface} becomes necessary if either $\gph N_C$ is regular around $(\bar{x},\bar{v})$ or $Q = \dom S$, (in the latter case the constraint qualification \cref{eqn:AVI-CQface} can be avoided.)
\end{theorem}

\begin{proof}
Noting that $\gph S|_Q$ is a union of finitely many polyhedral sets, for any $(q,x) \in \gph S|_Q$ sufficiently close to $(\bar{q},\bar{x})$, $(q',x') := (q-\bar{q}, x- \bar{x}) \in T_{\gph S|_Q}(\bar{q}, \bar{x})$ is sufficiently close to $(0,0)$. By \cref{eqn:AVI-gphofLQ} and \cite[Exercise 6.7, Theorem 6.42]{VaAn}, we have
\begin{equation}\label{eq:AVI-cq-CritFace-1}
  \begin{split}
    T_{\gph S|_Q}(\bar{q},\bar{x})  &\subseteq  \left( T_Q(\bar{q}) \times \R^n \right) \cap T_{\gph S} (\bar{q},\bar{x})  \\
     &= \left\{ (\tilde{q},\tilde{x}) \mid \tilde{q}\in T_Q(\bar{q}), \ (\tilde{x},-M\tilde{x}-\tilde{q}) \in T_{\gph N_C}(\bar{x},\bar{v}) \right\}. 
   \end{split}
\end{equation}
Therefore $q'\in T_Q(\bar{q})$ and $(x' , -Mx' -q') \in T_{\gph N_C}(\bar{x}, \bar{v})$.
By polyhedrality of $Q$, it holds that
\begin{equation}\label{eq:AVI-cq-CritFace0}
N_Q(\bar{q}+ q')  = N_Q(\bar{q}) \cap [q']^\perp.
\end{equation}
As $(\bar{q}, \bar{x})\in \gph S|_Q$, we have $(\bar{x}, -M\bar{x} - \bar{q}) \in \gph N_C$. Thus by \cref{eqn:AVI-faceexpGphNc}, we can derive the constraint qualification \cref{cq:AVI-ProjCde} as
\begin{equation}\label{eq:AVI-cq-CritFace}
    u^*\in  N_Q(\bar{q}), \   (M^* u^*, u^*)\in \left( F_1 -F_2 \right)^*  \times  \left( F_1 -F_2 \right) \Longrightarrow u^*=0,
\end{equation}
for some $F_1, F_2 \in \mathcal{F} (K(\bar{x},\bar{v}))$ with $F_2 \subset  F_1$. 

By using the upper estimate in \cref{prop:AVI-ProjCde-Expression} and the generalized Mordukhovich criterion in \cref{thm-genMordcri}, it would be sufficient to examine the Lipschitz-like property of $S$ relative to $Q$ at $(\bar{q}, \bar{x})$ by checking
\begin{equation}\label{eq:AVI-SufCondNBH1}
(M^* u^*, u^*) \in N_{\gph N_C}(x, -Mx-q), \ w^* \in N_Q(q) \Longrightarrow {\rm proj}_{T_Q(q)} (-u^* + w^*) =0,
\end{equation}
where $(q,x) \in \gph S|_Q$ is sufficiently close to $(\bar{q}, \bar{x})$.
By \cite[Exercise 12.22]{VaAn} and the polar relation between $T_Q(q)$ and $N_Q(q)$, the following equivalence holds,  
$${\rm proj}_{T_Q(q)} (-u^* + w^*) =0 \Longleftrightarrow -u^* + w^* \in N_Q(q), \ \mbox{ for any } w^* \in N_Q(q).$$ 
Noting $q = \bar{q} + q'$, it follows from \cref{eq:AVI-cq-CritFace0} and the convexity and conic structure of $N_Q(q)$ that 
 the above equivalence reduces to,
\begin{equation}\label{eq:AVI-cq-CritFace-1A}
 {\rm proj}_{T_Q(q)} (-u^* + w^*) =0, \mbox{ for any } w^* \in N_Q(q) \Longleftrightarrow  -u^* \in N_Q(q)  =N_Q(\bar{q}) \cap [q' ]^\perp.
\end{equation}
As $(q',x') \in T_{\gph S|_Q}(\bar{q}, \bar{x})$, by \cref{eq:AVI-cq-CritFace-1}, we have $(x', -Mx' - q') \in T_{\gph N_C}(x,v)$. Note $\bar{v}=-M\bar{x}-\bar{q}$. Let $v' = -Mx' - q'$. By \cref{lm:EquivforDirNCandNC}, we have
\begin{equation}\label{eq:AVI-cq-CritFace-1B}
    \begin{aligned}
        N_{\gph N_C}(x,-Mx-q) = \big\{ \left( F_1 -F_2 \right)^*  \times \left( F_1 -F_2 \right)  \mid \ &  x'\in F_2 \subset  F_1 \subset [v']^\perp, \\ 
        & F_1, F_2 \in \mathcal{F} (K(\bar{x},\bar{v})) \big\},
    \end{aligned}
\end{equation}  
where $(q,x) \in \gph S|_Q$ is sufficiently close to $(\bar{q}, \bar{x})$. Then by \cref{eq:AVI-SufCondNBH1}, \cref{eq:AVI-cq-CritFace-1A} and \cref{eq:AVI-cq-CritFace-1B}, a sufficient condition can be derived as
\begin{equation}\label{eq:sufcond-GenCritFaceNBH}
     \forall (M^* u^*, u^*)\in \left( F_1 - F_2 \right)^*  \times  \left(  F_1 - F_2 \right) \Longrightarrow u^* \in -N_Q(\bar{q})\cap [q']^\perp,
\end{equation} 
 for all closed faces $F_1,   F_2 \in \mathcal{F}(K(\bar{x}, \bar{v})) $ with $x'\in   F_2 \subset   F_1 \subset [-Mx'-q']^\perp$. 

 It remains to show that 
  the sufficient condition \cref{eq:sufcond-GenCritFaceNBH} can be equivalently replaced by the following condition:
     \begin{equation}\label{eq:sufcond-GenCritFace}
     \forall (M^* u^*, u^*)\in \left( F_1 -F_2 \right)^*  \times  \left( F_1 -F_2 \right) \Longrightarrow u^* \in -N_Q(\bar{q})
  \end{equation}
  for all closed faces $F_1, F_2 \in \mathcal{F} (K(\bar{x},\bar{v}))$ with $F_2 \subset  F_1$. 
  It is obvious that \cref{eq:sufcond-GenCritFaceNBH} implies \cref{eq:sufcond-GenCritFace} by taking $x'=  q' = 0$. Now we prove that \cref{eq:sufcond-GenCritFace} implies \cref{eq:sufcond-GenCritFaceNBH}. 
  For $F_1, F_2 \in \mathcal{F}(K(\bar{x}, \bar{v}))$ with $x'\in   F_2 \subset   F_1 \subset [-Mx'-q']^\perp$, we have
  $$ tx' \in F_2 \subset F_1 ,\  \forall t \geq 0$$ by the conic structure of $F_2$ and $F_1$. Thus we have
  $$[x'] \subset F_1 -F_2 \subset [-Mx'-q']^\perp. $$
  Note that $\left([-Mx'-q']^\perp\right)^* = [-Mx'-q'] $ and $ [x']^* = [x' ]^\perp$.
  Besides, $F_1- F_2$ is still a convex polyhedral cone and by the polar relation, we also have
  $$ [-Mx'-q'] \subset \left(F_1 -F_2 \right)^* \subset [x']^\perp. $$
 Therefore the following inclusions hold:
  $$(M^*u^*, u^*) \in  \left(F_1 -F_2 \right)^* \times \left(F_1 -F_2 \right)  \subset [x']^\perp \times [-Mx'-q']^\perp.$$
  From $M^* u^* \in (F_1 -F_2 )^* \subset [x']^\perp$, we have $ \langle  u^* , M x' \rangle =\langle M^* u^*  , x' \rangle = 0$. From $u^*\in F_1 -F_2 \subset [-Mx'-q']^\perp$, we have $\langle u^* , -q' \rangle =\langle u^* , -Mx'-q' \rangle =0$. Thus, $u^*\in [q']^\perp$ holds. Therefore \cref{eq:sufcond-GenCritFaceNBH} holds when  \cref{eq:sufcond-GenCritFace} is satisfied. 
  
  For any possible combinations of $F_1, F_2\in \mathcal{F} (K(\bar{x},\bar{v}))$ with $F_2 \subset F_1$, $F_1$ and $F_2$ are closed and convex cones and so is $F_1-F_2$. By \cite[Corollary 16.3.2]{CvxAn}, $M^* u^* \in (F_1-F_2)^* $ is equivalent to $ u^* \in (M(F_1-F_2))^*.$
  Then \cref{eqn:AVI-CQface} and \cref{eqn:AVI-SufCondface} can be derived from \cref{eq:AVI-cq-CritFace} and \cref{eq:sufcond-GenCritFace} respectively. 
  
Besides, by \cref{prop:AVI-ProjCde-Expression} and the convexity of $Q$, the upper estimate in \cref{eqn:AVI-ProjCode} becomes exact when either $\gph N_C$ is regular around $(\bar{x},\bar{v})$ or $Q = \dom S$. The inclusion \cref{eq:AVI-cq-CritFace-1} also becomes an equality when $\gph N_C$ is regular at $(\bar{x}, \bar{v})$. Then in these two cases, the sufficient condition becomes necessary. 
\end{proof}

Next we show that the Lipschitz-like property of a polyhedral multifunction relative to a convex set is equivalent to its relative inner semicontinuity. Before that, we give the definition and an alternative description of relative inner semicontinuity according to the hit-and-miss criteria. Here we adopt the version based on \cite{dontchev1996}, in which the inner semicontinuity is called as `lower semicontinuity'.

\begin{definition}[Inner semicontinuity relative to a set]\label{def:isc-relative}
Let $Q \subseteq \R^n$.
A set-valued mapping $S: \R^n \rightrightarrows \R^m$ is inner semicontinuous (isc) at $(\bar{q},\bar{x})\in \gph S|_Q$ relative to $Q$ if
  \begin{equation*}
     \bar{x} \in \liminf_{\tiny q\xrightarrow[]{Q}\bar{q}}S(q).
  \end{equation*}
  We say $S$ is isc around $(\bar{q},\bar{x})$ relative to $Q$ if there exists $V\times W \in \mathcal{N}(\bar{q})\times \mathcal{N}(\bar{x}) $ such that $S$ is isc at every $(q,x)\in \gph S|_Q \cap  (V\times W)$. 
\end{definition}
\begin{proposition}\label{prop:isc-equiv}
 $S$ is isc around $(\bar{q}, \bar{x})$ relative to $Q$ if and only if there exists $V\times W \in \mathcal{N}(\bar{q})\times \mathcal{N}(\bar{x}) $ such that for every open set $O$ with $S(q) \cap W \cap O \neq \emptyset$ where $q\in Q \cap V$, there exists $V'\in \mathcal{N}(q)$, such that $S(q') \cap O \neq \emptyset$ for all $q'\in Q\cap V'$.
\end{proposition}
\begin{proof}
 The multifunction $S$ being isc at every $(q,x)\in \gph S\cap (V\times W)$ relative to $Q$ is equivalent to that for every $x\in S(q) \cap W$ with $q\in Q\cap V$, $S$ is isc at $(q,x)$ relative to $Q$. That is $$S(q) \cap W \subset \liminf_{q'\xrightarrow[]{Q}q} S(q').$$
By the hit-and-miss criteria (\cite[Theorem 4.5]{VaAn}) we derive the statement.
\end{proof}

For a polyhedral multifunction, i.e., whose graph is a union of polyhedral sets, we know that it enjoys calmness at any point in its domain with the same $\kappa$ (see \cite[Proposition 1]{Robinson1981}). The equivalence between relative isc and relative Lipschitz continuity were given in \cite[Corollary 2.1]{robinson2007solution}. Here we present these equivalences around a given point. Note that this proof is very similar to that of \cite[Theorem 1.5]{robinson2007solution}. However, one cannot simply apply the result in \cite[Theorem 1.5]{robinson2007solution} as they adopted isc in a version without mentioning $\bar{x}$. This is different from our setting as for Lipschitz-like property, $\bar{x}$ is involved. 
For self-completeness, we include a proof here.
\begin{theorem}
Let $S:\R^n \rightrightarrows \R^m$ be a polyhedral multifunction and $\kappa$ be the calmness modulus of $S$ on $Q$, where $Q$ is a convex subset of $\dom S$. Consider $(\bar{q}, \bar{x}) \in \gph S|_Q$. Then for the following statements:
\begin{enumerate}[label=\normalfont(\roman*)]
    \item $S$ is isc around $(\bar{q}, \bar{x})$ relative to $Q$;
    \item $S$ is Lipschitz-like relative to $Q$ at $\bar{q}$ for $\bar{x}$ with modulus $\kappa$;
    \item $S$ is locally Lipschitz continuous relative to $Q$ around $(\bar{q}, \bar{x})$ with modulus $\kappa$;
\end{enumerate}
we have {\normalfont(i) $\Longleftrightarrow$ (ii) $\Longleftarrow$ (iii)}. 
\end{theorem}
\begin{proof}
It is obvious that (iii) $\Longrightarrow$ (ii).
For proving {(i) $\Longrightarrow$ (ii)}, we use the equivalent statement 
of isc in \cref{prop:isc-equiv}, i.e.,  there exists $V\times W \in \mathcal{N}(\bar{q})\times \mathcal{N}(\bar{x}) $ such that for every open set $O$ with $S(q) \cap W \cap O \neq \emptyset$ where $q\in Q \cap V$, there exists $V'\in \mathcal{N}(q)$, such that $S(q') \cap O \neq \emptyset$ for all $q'\in Q\cap V'$. Choose any two different points $q_0, q_1 \in Q\cap V$. For each $t\in (0,1)$ we define $q_t:= (1-t) q_0 + t q_1$. Given that $Q\subset \dom S$ is convex, $S$ is calm at $q_t$ (relative to $Q$) with $\kappa$ for all $t\in [0,1]$, i.e., $\exists \rho_t >0 $ such that 
$$S(q') \subseteq S(q_t) + \kappa \| q' - q_t \| \B, \ \forall q' \in Q\cap \B_{\rho_t}(q_t). $$
We define 
\begin{equation}\label{eqn:pf-def-tau}
    \tau := \sup \left\{ t \in [0,1] \mid \text{ for each }s\in [0,t], S(q_s ) \subseteq S(q_0) + \kappa \| q_s - q_0 \| \B \right\}.
\end{equation} Then we have $\tau >0$ as $\rho_0 > 0$. Next we would like to show that
\begin{equation}\label{eqn:pf-ap-key}
    S(q_\tau) \cap W \subset S(q_0) + \kappa \|q_\tau - q_0 \| \B.
\end{equation}
Suppose on the contrary the inclusion \eqref{eqn:pf-ap-key} does not hold. Given that $S(q_0)$ is closed and so is the right-hand side of \eqref{eqn:pf-ap-key}, by letting $O:= \overline{ S(q_0) + \kappa \|q_\tau - x_0 \| \B}$ (here a bar means a complement of the set), we have that $O$ is open and $S(q_\tau) \cap W \cap O \neq \emptyset$. Note that $q_\tau \in Q \cap V$. By isc of $S$ around $(\bar{q}, \bar{x})$, there exists $\sigma \in [0, \tau)$ such that $S(q_\sigma)$ meets $O$. However, by definition of $\tau$, see \eqref{eqn:pf-def-tau}, all such $\sigma$ satisfy 
$$ S(q_\sigma ) \subseteq S(q_0) + \kappa \| q_\sigma - q_0 \| \B \subset S(q_0 ) + \kappa \| q_\tau - q_0 \| \B, $$ which contradicts the fact that $S(q_\sigma)$ meets $O$ and therefore \eqref{eqn:pf-ap-key} holds. 

Next we prove that $\tau= 1$. Suppose on the contrary $\tau <1$. By definition of $\tau$, then there exists $\phi \in (\tau, 1) $ with $\|q_\phi - q_\tau\| < \rho_\tau$, such that 
\begin{equation}\label{eqn:tau-1-contra}
    S(q_\phi) \not\subset S(q_0 ) + \kappa \|q_\phi -  q_0\| \B.
\end{equation}
By definition of $\rho_\tau$ and the calmness of $S$ at $q_\tau$,
$S(q_\phi) \subset S(q_\tau) + \kappa \|q_\phi -q_\tau \| \B$. 
Considering \eqref{eqn:pf-def-tau} and the fact that $q_0, q_\tau, q_\phi$ are collinear, we further have 
\begin{equation*}
    \begin{aligned}
    S(q_\phi) \cap W \subset S(q_\phi) \subset S(q_\tau ) + \kappa \|q_\phi - q_\tau\| \B & \subset S(q_0) + \kappa (\|q_\phi - q_\tau\| + \|q_\tau -q_0\| )  \B \\
    & = S(q_0) + \kappa (\|q_\phi  -q_0\| )  \B,
    \end{aligned}
\end{equation*}
which contradicts \eqref{eqn:tau-1-contra} and therefore $\tau = 1$. Then we have 
$S(q_1) \cap W \subset S(q_0 ) + \kappa \| q_1 - q_0 \| \B$ and by switching $q_0$ and $q_1$ we complete the proof in this direction. 

For (ii) $\Longrightarrow$(i), we know that there exist $V\in \mathcal{N}(\bar{q})$ and $W \in \mathcal{N}(\bar{x})$, such that the inclusion 
$$S(q) \cap W \subset S(q') + \kappa \|q' - q\|\B, \ \forall q, q'\in Q\cap V. $$
Let $O$ be any open set such that $S(q) \cap W \cap O \neq \emptyset$ for any $q\in Q \cap V$. Then there exists $x\in S(q)\cap W\cap O$ and $\eta >0 $ such that $\B_\eta(x) \subset O$. Also let $\nu >0 $ be a number satisfying $\kappa \nu < \eta$ and $\B_{\nu} (q)  \subset V$. 
Then we have
$$x \in S(q)\cap W  \subset S(q') + \kappa \|q-q'\|\B \subset S(q') +\eta \B,\  \forall q'\in Q \cap \B_\nu (q). $$
That is, there exists $x' \in S(q')$ with $x' \in \B_{\eta}(x)$ and therefore $ x'\in S(q') \cap O \neq \emptyset$ and $S$ is isc around $(\bar{q}, \bar{x})$ relative to $Q$.
\end{proof}

In \cite[Theorem 1]{dontchev1996} and the references therein, a result on equivalences among isc, Lipschitz-like property and single-valuedness along with Lipschitz continuity is given for the solution mapping of AVI. However, in the next example, we illustrate that the solution mapping $S(q)$ \cref{eqn:AVI} is not Lipschitz-like relative to its domain, but is so on a polyhedral subset and that the local single-valuedness is not be valid with the relative Lipschitz-like property. 

\begin{example}\label{ex:lcp-AprtQ}
Consider the solution mapping $S:\R^2\rightrightarrows \R^2$ of the AVI \cref{eqn:AVI}
with
$
 M = \left(\begin{smallmatrix}
                -1 & -1 \\
                1 & -1
              \end{smallmatrix}
              \right) $ and $ C = \R^2_+$.
Then we have, for $q = (q_1, q_2) \in \R^2$, 
\begin{equation}\label{eqn:AVI-example-S}
    S(q) = \begin{cases}
\left\{ (0,0), \ (q_1, 0) \right\}, &  0 \leq q_1 < q_2,\\
\left\{ (0,0), \ (q_1, 0) , \ (0, q_2) , \ (\frac{q_1 - q_2}{2}, \frac{q_1 + q_2}{2}) \right\},   & 0 \leq q_2 \leq q_1, \\
\left\{  (q_1, 0) ,  \ (\frac{q_1 - q_2}{2}, \frac{q_1 + q_2}{2}) \right\},  & -q_1 \leq q_2 < 0,\\
\emptyset, & \mbox{otherwise}. \end{cases}
\end{equation}
Thus $\dom S=\left\{ q\in \R^2 \mid  q_2 \geq - q_1, \ q_1 \geq 0\right\} $, which is a convex set. Let $(\bar{q}, \bar{x})=(0_2, 0_2)\in \gph S$. However $S$ does not enjoy the Lipschitz-like property relative to $\dom S$ at $\bar{q}$ for $\bar{x}$.
Indeed, for a fixed and small enough $\varepsilon >0$, let $q = (\varepsilon, \varepsilon)$ and $q' = (\varepsilon, \varepsilon + \varepsilon') $ with arbitrarily small $\varepsilon'$ with $0< \varepsilon'<\varepsilon $. Then we have $q, q' \in \dom S \cap \B_{2\varepsilon}(\bar{q})$. 
By \cref{eqn:AVI-example-S}, $S(q) = \left\{(0,0), (0,\varepsilon),(\varepsilon,0) \right\}$ and $S(q') = \left\{(0,0), (\varepsilon,0) \right\}$. 
For $$S(q)\cap \B_{2\varepsilon}(0) = S(q) \subset S(q') + \kappa \varepsilon' \B,$$ 
to hold, in particular
$$(0,\varepsilon) \in \left\{(0,0), (\varepsilon,0) \right\}+ \kappa \varepsilon' \B,$$ 
it is required $\varepsilon' \geq \frac{\varepsilon }{\kappa}$ or $\varepsilon' \geq \frac{\sqrt{2}\varepsilon }{\kappa}$. But this cannot be true when $\varepsilon'$ is sufficiently small.

Let 
$Q = \left\{ q\in \R^2 \mid  0\leq q_2 \leq q_1\right\} \subset \dom S,$  which is a polyhedral set. 
We also have $(\bar{q}, \bar{x}) \in \gph S|_Q$. Then $$N_Q(\bar{q}) = \left\{v^*\in \R^2 \mid v^*_2\leq  -v^*_1, v^*_1 \leq 0\right\}.$$
The critical cone $K(\bar{x}, -M\bar{x} - \bar{q}) = T_{\R^2_+}(\bar{x}) \cap [-M\bar{x} - \bar{q}]^\perp  = \R^2_+$. Then 
$$\mathcal{F}(K(\bar{x}, -M\bar{x} - \bar{q})) = \{0_2\} \cup \left(\R_+\times \{0\}\right) \cup  \left(\{0\} \times \R_+\right) \cup \R^2_+.$$
Therefore for all $F_2\subseteq F_1 \in \mathcal{F}(K(\bar{x}, -M\bar{x} - \bar{q}))$, the possible $F_1 - F_2 $ are: $\{0_2\},  \ \R_+ \times \{0\}, \ \R \times \{0\}, \ \{0\} \times \R_+ ,\  \{0\} \times \R, \ \R^2_+, \ \R \times \R_+, \ \R_+ \times \R , \ \R^2$. 
Let $( M^*u^*, u^*) \in (F_1 - F_2)^* \times (F_1 - F_2)$. It generates 
\begin{equation*}
        u^* \in \left(\{0\} \times \R_+ \right) \cup \left(\R_+ \times \{0\}\right) \cup \R_+ (1, -1)   \cup \left\{u^*\in \R^2  \mid 0 \leq u^*_2 \leq u^*_1\right\}  \subset -N_Q(\bar{q}).
\end{equation*}
Thus both the constraint qualification \cref{eqn:AVI-CQface} and the sufficient condition \cref{eqn:AVI-SufCondface} are satisfied. 
Therefore $S$ is Lipschitz-like relative to $Q$ at $\bar{q} = 0$ for $\bar{x} = 0$. Besides, we can also see that $S$ is not locally single-valued around $(\bar{q},\bar{x}) $ on $Q$.

\end{example}

\section{Linear complementarity problems with $Q_0$-matrices}\label{sect:LCP}
 Consider the following linear complementarity problem (in short LCP): 
\begin{equation}\label{eqn:LCP-def}
  x \geq 0,\ Mx + q \geq 0,\ x^\top (Mx + q) = 0
\end{equation}
where $q, x\in \R^n \text{ and } M$ is an $n\times n$ matrix. This is also a type of AVI with the polyhedral set $C$ in \cref{eqn:AVI} being $\R^n_+$. We denote the solution mapping of \cref{eqn:LCP-def} with $q$ being the parameter as $S$. With some specific structure of matrix $M$, we can transform the generalized critical face condition obtained in section \ref{sect:AVI} into a more explicit one by using index related notations and obtain a necessary and sufficient condition without any regularity condition.

First we write the domain and graph of $S$ respectively as follows:
\begin{align}
    \dom S= &  \left\{ q \in \R^n \mid \exists x \geq 0,\; Mx + q \geq 0,\; \langle x, \, Mx+q \rangle =0 \right\}, \nonumber\\ 
    \gph S= &  \left\{ (q, x)\in \R^n\times \R^n \mid x \geq 0,\; Mx + q \geq 0,\; \langle x, \, Mx+q \rangle =0 \right\}.\label{eq:def:LCPgph}
\end{align}
Here we introduce some properties of the sets $\dom S$ and $\gph S$. 
Let $I := \left\{1, \dots, n\right\}$. To analyze $\gph S$, we first 
denote the set of all possible index combinations of $I$ as:
\[
\mathcal{I}:=\{(I_1, I_2, I_3)\mid  I_1\cup I_2\cup I_3=I,\;I_i\cap I_j=\emptyset, \ \forall i\neq j\}, \
| \mathcal {I}|= 3^n.
\]
In terms of each combination $ (I_1, I_2, I_3)\in \mathcal{I}$, we denote:
\begin{equation}\label{eqn:LCP-gphslice}
(\gph S)_{(I_1, I_2, I_3)}:=\left\{
(q, x)\in \R^n\times \R^n\left|  \begin{array}{lll}
  x_i=0 & (Mx+q)_i>0 &\mbox{if}\; i\in I_1 \\
  x_i>0 & (Mx+q)_i=0 &\mbox{if}\; i\in I_2\\
  x_i=0 & (Mx+q)_i=0 &\mbox{if}\; i\in I_3\\
\end{array}
\right.
\right\}.
\end{equation}
We call $(\gph S)_{(I_1, I_2, I_3)}$ a slice of $\gph S$.
 Recall that {\it a semi-closed polyhedron}, which is originated from \cite{yang2010structure} (see also \cite{fang2015minimal}), 
is defined as the intersection of finitely many closed or open half-spaces. 
\begin{proposition}\label{gphSproperty}
The slice $(\gph S)_{(\cdot, \cdot, \cdot)}$ has the following properties:
  \begin{description}
  \item[(a)] each slice of $\gph S$ is uniquely determined by its index combination and all the slices of $\gph S$ are mutually exclusive, that is, $\forall (I_1, I_2, I_3)\not=(I'_1, I'_2, I'_3),$
  $$(\gph S)_{(I_1, I_2, I_3)} \not= (\gph S)_{(I'_1, I'_2, I'_3)}, \mbox{ and } (\gph S)_{(I_1, I_2, I_3)}\cap (\gph S)_{(I'_1, I'_2, I'_3)}=\emptyset; $$
  \item[(b)] $\gph S$ is a union of all the slices $(\gph S)_{(\cdot,\cdot, \cdot)}$ where the index combination runs through all elements in $\mathcal{I}$:
        \begin{equation*}
          \gph S=\displaystyle\bigcup_{(I_1, I_2, I_3)\in \mathcal{I}}(\gph S)_{(I_1, I_2, I_3)};
        \end{equation*}
  \item[(c)] for each $(I_1, I_2, I_3)\in \mathcal{I}$, $(\gph S)_{(I_1, I_2, I_3)}$ is a nonempty, convex and semi-closed polyhedral cone. Here we adopt the definition of cone $C$ as follows:
$$\forall x \in C , \   \lambda \in \R_{++} \Longrightarrow \lambda x \in C.$$ 
  \end{description}
\end{proposition}
\begin{proof}
   The first two properties and the polyhedrality in the third property can be observed by checking the definitions of graph $\gph S$, \cref{eq:def:LCPgph}, and slice $(\gph S)_{(I_1, I_2, I_3)}$, \cref{eqn:LCP-gphslice}. 
       For the conic structure and convexity mentioned in the third property, letting $(q_1,x_1), (q_2, x_2) \in (\gph S)_{(I_1, I_2, I_3)}$, and $\lambda_1, \lambda_2 >0$, we have $\lambda_1 (q_1,x_1) + \lambda_2(q_2,x_2) \in (\gph S)_{(I_1, I_2, I_3)}$ as elements $(\lambda_1 x_1 + \lambda_2 x_2)_i,\ \lambda_1 (Mx_1+q_1)_i + \lambda_2 (M x_2 + q_2)_i, \text{ for } i \in \left\{1, \dots, n\right\}$ remain their status of being $0$ or positive. 
 \end{proof}

Depending on the combination above, the following proposition gives the interpretation of the index sets $I_1, I_2, I_3$ on representation of $q$. Note that we use $E$ as an identity matrix.

\begin{proposition}[{\cite[Proposition 1.4.4]{LCP}}] \label{qrep}
  Any $q \in \dom S$ can be expressed as a positive linear combination of some columns in $E$ and $-M$ as
  \begin{equation}\label{eqn:LCP-qlinearcomb}
  q = \sum_{i: (q+Mx)_i > 0 } E_{(\cdot, i)} (q+Mx)_i  - \sum_{i : x_i > 0} M_{(\cdot, i)} x_i,
\end{equation}
where $x \in S(q)$. 
\end{proposition}

From Proposition \ref{gphSproperty} (a) and Proposition \ref{qrep}, we can see that for $(q, x)\in \gph S$, there is a unique index combination $(I_1, I_2, I_3)$ $\in \mathcal{I}$ such that $(q, x)\in$ $(\gph S)_{(I_1, I_2, I_3)}$. To avoid abuse of notations, we specify the unique index combination decided by $(q,x)$ as
\begin{align}
  I_1(q, x):=& \{ i\in I | x_i = 0, (Mx+q)_i >0\}, \nonumber\\
  I_2(q, x) :=&  \{ i\in I | x_i > 0, (Mx+q)_i =0\}, \label{eqn:LCP-idxset-qx}\\
  I_3(q, x) :=&\{ i\in I | x_i = 0, (Mx+q)_i = 0\}. \nonumber
\end{align}
Thus by the representation above we can see that when $(q,x)$ is given, for any pair $(q',x') \in (\gph S)_{(I_1 (q, x), I_2 (q, x), I_3 (q, x))}$, the index combination $(I_1(q',x')$, $I_2(q',x')$, $I_3(q',x'))$ remains the same as the one of $(q, x)$.
With the pair $(q,x)\in \gph S$ fixed, we may now proceed to the representation of $ N_{\gph S}(q, x)$. To better illustrate the structure of $ N_{\gph S}(q, x)$, we introduce a set defined by index combinations:
\begin{equation}\label{eqn:LCP-def-W}
W(I_1, I_2, I_3) := \left\{ (u^*, v^*) \in \R^n \times \R^n \left|
  \begin{array}{ll}
  (u^*_i,v^*_i)\in \{0\}\times \R &\mbox{if}\; i\in I_1 \\
(u^*_i,v^*_i)\in \R\times \{0\}&\mbox{if}\; i\in I_2  \\
  (u^*_i,v^*_i)\in  \Omega & \mbox{if}\; i\in I_3
\end{array} \right. \right\}
\end{equation}
where $\Omega:=  (\{0\} \times \R) \cup (\R\times \{0\})\cup \R^2_-$.
Note that for $(q,x)\in \gph S$, 
\begin{equation}\label{equiv:WidxVSNgphNR+}
(u^*, v^*)\in W(I_1(q,x), I_2(q,x), I_3(q,x)) \Longleftrightarrow (v^*, -u^*) \in N_{\gph N_{\R^n_+}}(x, -Mx-q). 
\end{equation}
With such an equivalence we can see that the regularity of $\gph N_{\R^n_+}$ at $(x, -Mx-q)$ requires $I_3(q,x) = \emptyset.$
By calculation in \cite{Huyen2019} we have
\begin{equation} \label{eqn:LCP-NCofgphS}
N_{\gph S}(q, x)=\left\{(u^*,\;M^* u^*+v^*) \mid (u^*, v^*)\in W(I_1(q,x), I_2(q,x), I_3(q,x)) \right\}.
\end{equation}
From \cref{eqn:LCP-NCofgphS}, we can see that the normal cone of $\gph S$ at a given pair is decided by the associated index combination.  Given the discussion above that the index combination remains unchanged for all elements in the slice $(\gph S)_{(I_1,I_2, I_3)}$, in accordance $N_{\gph S}(q, x)$ stays the same for all $ (q,x) \in (\gph S)_{(I_1, I_2, I_3)}$ as well. In this way, we can use the index combination to recognize the behavior of neighboring points as $(q,x)\xrightarrow[]{\gph S}(\bar{q}, \bar{x})$ and the related $N_{\gph S}(q,x)$. Next result shows how we can group the elements by index combination.
\begin{proposition}
For $(q^k, x^k)\overset{\gph S}\longrightarrow (\bar{q}, \bar{x})$, there is a subsequence $(q^{ki}, x^{ki})$ such that the index combination $I_1(q^{ki}, x^{ki})$, $I_2(q^{ki}, x^{ki})$, $I_3(q^{ki}, x^{ki})$ categorized as in \cref{eqn:LCP-idxset-qx} remain the same for all $i$.
\end{proposition}
\begin{proof}
Let sequence $(q^k, x^k)\longrightarrow (\bar{q}, \bar{x})$ in $\gph S$. For each $k$, $(q^k, x^k)$ has a corresponding combination of index set $I_1(q^k, x^k) ,I_2(q^k, x^k) ,I_3(q^k, x^k)$. As such combinations of index set is finite, there is a subsequence $(q^{ki}, x^{ki})$ such that the corresponding combination of index sets $I_1(q^{ki}, x^{ki})$, $I_2(q^{ki}, x^{ki})$, $I_3(q^{ki}, x^{ki})$ remains the same. Then the normal cone $N_{\gph S}(q^{ki}, x^{ki})$ remains the same as well.
\end{proof}

Given the pair $(\bar{q},\bar{x}) \in \gph S$, we denote a collection of combinations of index sets decided by $(\bar{q},\bar{x})$ as:
\begin{equation}\label{eqn:LCP-NBHidxcomb}
  \mathcal{I}(\bar{q}, \bar{x})  = \left\{ (I_1, I_2, I_3) \in \mathcal{I} \ | \  I_1 \supseteq I_1 (\bar{q},\bar{x}),  I_2 \supseteq I_2 (\bar{q},\bar{x}), I_3 \subseteq I_3 (\bar{q},\bar{x}) \right\}.
\end{equation}
As mentioned in the proof above, there are finite combinations of index sets around $(\bar{q}, \bar{x})$. Next we illustrate what those combinations and neighboring slices are. 
\begin{lemma}\label{thm:LCP-NBHslices}
Given the pair $(\bar{q},\bar{x}) \in \gph S$,
the corresponding neighboring slices in $\gph S$ are finite as: $$(\gph S)_{(I_1, I_2, I_3) }, \ \forall (I_1, I_2, I_3)\in \mathcal{I}(\bar{q}, \bar{x}).$$
\end{lemma}
\begin{proof}
  For $(\bar{q}, \bar{x})\in \gph S$, specifically it lies on slice $(\gph S)_{(I_1 (\bar{q},\bar{x}), I_2 (\bar{q},\bar{x}) , I_3 (\bar{q},\bar{x}))}$. Consider $(q,x)$ around $(\bar{q}, \bar{x}) $ in $\gph S$. By \cref{eqn:LCP-idxset-qx}, we can see that for $I_1 (q,x), I_2 (q,x)$, there are open constraints $(Mx+q)_i >0$ and $x_i >0$ respectively and for $I_3 (q,x)$, there are equalities: $(Mx+q)_i =0 \text{ and }  x_i =0$. Therefore for any element $(q,x)\xrightarrow[]{\gph S}(\bar{q},\bar{x})$, $I_1(q,x) \text{ and } I_2(q,x)$  should include $I_1 (\bar{q},\bar{x}) \text{ and } I_2 (\bar{q},\bar{x})$ as subsets respectively. For equality constraints, it can be tended through open constraints, either $(Mx+q)_i >0$ or $x_i>0$. Then we have the indices in $I_3(\bar{q},\bar{x})$ being distributed into either $I_1(q,x)$, $I_2(q,x)$ or remained in $I_3(q,x)$. Thus $I_3(q,x) \subseteq I_3(\bar{q},\bar{x})$. Combining all these finite possible statuses, we arrive at \cref{eqn:LCP-NBHidxcomb}. With possible combinations decided, we can accordingly give neighboring slices as stated.
\end{proof}

With index combination of neighboring slices given, we can see that for a given pair $(\bar{q}, \bar{x})$, there are $| \mathcal{I}(\bar{q},\bar{x})| = 3^{|I_3(\bar{q},\bar{x})|}$ neighboring slices in $\gph S$ (including the slice where the pair lies on). For example, for $(\bar{q}, \bar{x}) = (0,0)$, $I_1 (\bar{q}, \bar{x}) =  \emptyset$, $I_2 (\bar{q}, \bar{x}) = \emptyset$, $I_3 (\bar{q}, \bar{x}) = I$. Then $\mathcal{I}(\bar{q}, \bar{x})$ gives all possible combinations of $I_1, I_2, I_3$: $\mathcal{I}(\bar{q}, \bar{x}) = \mathcal{I}$ and $| \mathcal{I}(\bar{q},\bar{x})| = |\mathcal{I}| = 3^n. $ The neighboring slices are all slices of $\gph S$.

As the assumption of the generalized Mordukhovich criterion requires the relative set to be closed and convex, it is natural to ask under what condition $\dom S$ is closed and convex. The following proposition provides the rationality behind such an assumption. 

\begin{lemma}[{\cite[Proposition 3.2.1]{LCP}}] \label{prop:domScvxcond}
For an $\text{LCP}(q,M)$ defined as in \cref{eqn:LCP-def}, the following statements are equivalent:
\begin{enumerate}[label=\normalfont(\roman*)]
    \item $M$ is a $Q_0$-matrix.
    \item $\dom S$ is a polyhedral cone in $\R^n$;
    \item $\dom S = \conv \pos (E,-M).$
\end{enumerate} 

Here a $Q_0$-matrix means the type of matrices with LCP\cref{eqn:LCP-def} being solvable whenever feasible.
\end{lemma}

The expression in Lemma \ref{prop:domScvxcond} (iii) allows one to give explicit expressions of normal cones and tangent cones of $\dom S$.
\begin{proposition}\label{prop:LCP-domTCNC}
  For the solution mapping $S$ of \cref{eqn:LCP-def}, let $M$ be a $Q_0$-matrix. For a given combination of index set $(I_1, I_2, I_3)$, any $q$ with $(q,x) \in (\gph S)_{(I_1, I_2, I_3)}$ for some $x$ has the following properties:
  \begin{equation}\label{eqn:LCP-domNC} 
      N_{\dom S}(q) = N(I_1, I_2, I_3)  := \left\{w \ \bigg| \
  \begin{matrix}
    \langle w, v \rangle = 0,  &  v\in \left(E_{(\cdot, I_1)}, -M_{(\cdot, I_2)}\right) \\
    \langle w, u \rangle \leq 0,  &  u\in \left(E_{(\cdot, \overline{I_1})}, -M_{(\cdot, \overline{I_2})}\right)
  \end{matrix}
    \right\},
  \end{equation}
  \begin{equation*}
        T_{\dom S}(q)= T(I_1, I_2, I_3)  :=  \conv \pos \left( E_{(\cdot, \overline{I_1})}, \ -M_{(\cdot, \overline{I_2})},\ \pm E_{(\cdot, I_1)}, \ \pm M_{(\cdot, I_2)} \right). 
  \end{equation*}
  Here $\overline{I_1} := I\backslash I_1 = I_2 \cup I_3$ and $v\in \left(E_{(\cdot, I_1)}, -M_{(\cdot, I_2)}\right)$ means $v$ being the column vector chosen from the matrix $\left(E_{(\cdot, I_1)}, -M_{(\cdot, I_2)}\right)$.
\end{proposition}
\begin{proof}
  By \cref{prop:domScvxcond}, $\dom S$ is a polyhedral cone and so are $T_{\dom S}(q)$ and $N_{\dom S}(q)$. By \cite[Proposition 2A.3]{dontchev2009implicit},
  $$w\in N_{\dom S} (q) \Longleftrightarrow q\in \dom S, \ w\in (\dom S)^*, w\perp q. $$
  As $\dom S = \conv \pos (E, -M)$, by \cite[Lemma 6.45]{VaAn}, we have
  $$(\dom S)^* = \left\{w \ | \ \langle E_{(\cdot, i)},w \rangle \leq 0, \langle -M_{(\cdot, i)},w \rangle \leq 0,  i = 1, \dots, n \right\}.   $$
  Note that $(q,x) \in (\gph S)_{(I_1, I_2, I_3)}$, by representation \cref{eqn:LCP-qlinearcomb} of $q$, $q$ can be expressed as a positive linear combination of $E_{(\cdot, i)}, i\in I_1$ and $-M_{(\cdot, j)}, j\in I_2$. Therefore we have
  $$N_{\dom S}(q) = (\dom S)^* \cap [q]^\perp = \left\{w  \bigg| \
  \begin{matrix}
    \langle E_{(\cdot, i)},w \rangle \leq 0, i\in \overline{I_1},  & \langle -M_{(\cdot, j)},w \rangle \leq 0, j\in \overline{I_2} \\
     \langle E_{(\cdot, i)},w \rangle = 0, i\in I_1 , & \langle -M_{(\cdot, j)},w \rangle = 0, j\in I_2
  \end{matrix}
    \right\} . $$
 From the polar relation between a normal cone and a tangent cone of a convex set, 
we can derive:
  $$T_{\dom S}(q) = (N_{\dom S}(q))^* = \conv \pos \left( E_{(\cdot, \overline{I_1})},\ -M_{(\cdot, \overline{I_2})}, \  \pm E_{(\cdot, I_1)}, \ \pm M_{(\cdot, I_2)} \right).  $$
The proof is complete.
\end{proof}

From the theorem above we can see that the tangent cone and the normal cone stay the same for all $(q,x)$ on the same slice of $\gph S$ as long as the index combination is fixed. Note that when only $q$ is given, without index combination, the linear combination in \cref{eqn:LCP-qlinearcomb} is not unique.

Following \cref{thm:genCritFace}, in next theorem we prove that under some specific setting we can use only the information at the given point to obtain a sufficient and necessary condition for the relative Lipschitz-like property. Before presenting the condition, we introduce another set defined by index combination similar to $W(I_1, I_2,I_3)$:
\begin{equation}
W'(I_1, I_2, I_3) := \left\{ (u^*, v^*) \in \R^n \times \R^n \left|
  \begin{array}{ll}
  (u^*_i,v^*_i)\in \{0\}\times \R_- &\mbox{if}\; i\in I_1 \\
(u^*_i,v^*_i)\in \R_- \times \{0\}&\mbox{if}\; i\in I_2 \label{eqn:LCP-WprimeidxDef} \\
  (u^*_i,v^*_i)\in  \R^2_- & \mbox{if}\; i\in I_3
\end{array} \right. \right\}. 
\end{equation}
Note that this set is generated by replacing all $\R$ with $\R_-$ in $W(I_1, I_2, I_3)$.
\begin{theorem} \label{thm:LCP-sufneccond}
For LCP\cref{eqn:LCP-def} with $M$ being a $Q_0$-matrix, let $(\bar{q}, \bar{x}) \in \gph S$. The solution mapping $S$ has the Lipschitz-like property relative to its domain at $\bar{q}$ for $\bar{x}$ if and only if
 \begin{equation}\label{eqn:sufnecsinglepointwithW}
     \begin{aligned}
        & \forall (u^*, -M^* u^* ) \in  W(I_1(\bar{q} ,\bar{x}), I_2(\bar{q} ,\bar{x}),I_3(\bar{q} ,\bar{x})) \Longrightarrow \\
        & \hskip4cm (u^*, -M^* u^* ) \in W'(I_1(\bar{q} ,\bar{x}), I_2(\bar{q} ,\bar{x}),I_3(\bar{q} ,\bar{x})).
       \end{aligned}
 \end{equation}
\end{theorem}
\begin{proof}
   Given \cref{prop:domScvxcond} and $M$ being a $Q_0$-matrix, $\dom S$ is a polyhedral cone. 
   Then the sufficient condition \cref{eqn:AVI-SufCondface} in \cref{thm:genCritFace} is also necessary in this case.
   Let $\bar{v} =  -M\bar{x}- \bar{q} \in N_{\R^n_+}(\bar{x})$. Recall that $K(\bar{x}, \bar{v}) = T_{\R^n_+} (\bar{x}) \cap [\bar{v}]^\perp.$ Considering the possible signs of each pair $(\bar{x}_i, \bar{v}_i)$, we have the following component-wise representation of the critical cone $K(\bar{x}, \bar{v})$ by virtue of the associated index combination in \cref{eqn:LCP-idxset-qx}:
  $$K(\bar{x}, \bar{v}) = \{0_{I_1(\bar{q}, \bar{x})} \} \times \R^{I_2(\bar{q}, \bar{x})}  \times \R^{I_3(\bar{q}, \bar{x})}_+.$$
 As the closed face of $K(\bar{x},\bar{v})$ comes in the form of $F = K(\bar{x},\bar{v})\cap [v^*]^\perp$ where $v^*\in K(\bar{x},\bar{v})^*$ and the possible closed faces of all the components of $K(\bar{x},\bar{v})$ being:
  $$\mathcal{F}(0) = \left\{ \{0 \} \right\} ,   \ \mathcal{F}(\R) = \{\R\}, \ \mathcal{F}(\R_+) =\left\{\{0\}, \R_+ \right\},$$
we thus have
$$\mathcal{F}(K(\bar{x}, \bar{v})) =   \{0_{I_1(\bar{q},\bar{x}) \cup I'}\} \times \R^{I_2(\bar{q},\bar{x})} \times \R_+^{I_3(\bar{q},\bar{x}) \backslash I'},
$$
 where $I' \subset I_3(\bar{q},\bar{x})$. 
Let $F_1, F_2$ be any two faces of $\mathcal{F}(K(\bar{x}, \bar{v}))$ with $F_2 \subseteq F_1$. By considering all the combinations of algebraic operations of three sets $\{0\}, \R_+$ and $\R$, we further have, for some subset $I'' \subset I_3(\bar{q},\bar{x}) \backslash I'$, 
\begin{equation} \label{eqn:sufnecsinglepointdiffcf}
F_1 - F_2 = \{0_{I_1(\bar{q},\bar{x}) \cup I'}\} \times \R^{I_2(\bar{q},\bar{x})\cup I''} \times \R^{I_3(\bar{q},\bar{x})\backslash (I'\cup I'')}_+.
\end{equation}
Therefore
$$\bigcup_{\substack{F_2 \subseteq F_1\\ F_1, F_2 \in \mathcal{F}(K(\bar{x}, \bar{v}))}} F_1 - F_2 =   \bigcup_{\substack{I', I''\subseteq I_3(\bar{q},\bar{x}) \\I'\cap I'' = \emptyset} } \{0_{I_1(\bar{q},\bar{x}) \cup I'}\} \times \R^{I_2(\bar{q},\bar{x})\cup I''} \times \R^{I_3(\bar{q},\bar{x})\backslash (I'\cup I'')}_+.$$ 
And noting also that the polar sets of three sets $\{0\}, \R$ and $\R_+$ are respectively
$$\{0\}^* = \R, \ \R^* = \{0\}, \ (\R_+)^* = \R_-,$$
it follows from \cref{eqn:sufnecsinglepointdiffcf} that we have
$$\bigcup_{\substack{F_2 \subseteq F_1\\ F_1, F_2 \in \mathcal{F}(K(\bar{x}, \bar{v}))}} \left( F_1 - F_2 \right)^* =   \bigcup_{\substack{I', I''\subseteq I_3(\bar{q},\bar{x}) \\I'\cap I'' = \emptyset} } \R^{I_1(\bar{q},\bar{x}) \cup I'} \times \{0\}_{I_2(\bar{q},\bar{x})\cup I''} \times \R^{I_3(\bar{q},\bar{x})\backslash (I'\cup I'')}_-.$$ 

Therefore by definition \cref{eqn:LCP-def-W} we can see that 
\begin{equation}\label{equiv:LCP-facetoW}
    \begin{aligned}
       (M^* u^*, u^* ) & \in \bigcup_{\substack{F_2 \subseteq F_1\\ F_1, F_2 \in \mathcal{F}(K(\bar{x}, \bar{v}))}} \bigg( (F_1 - F_2)^* \times (F_1-F_2) \bigg) \\
        & \hskip2cm \Longleftrightarrow  (-u^*, M^* u^*) \in  W(I_1(\bar{q}, \bar{x}), I_2(\bar{q}, \bar{x}), I_3(\bar{q}, \bar{x})).
    \end{aligned}
\end{equation}

Therefore, by adjusting the direction of $u^*$, the condition \cref{eqn:AVI-SufCondface} is equivalent to 
$$(u^*, -M^*u^*) \in W(I_1(\bar{q}, \bar{x}), I_2(\bar{q}, \bar{x}), I_3(\bar{q}, \bar{x}))\Longrightarrow u^* \in N_{\dom S}(\bar{q}).$$
As $\langle u^*,  E_{(\cdot, i)} \rangle =  u^*_i$ and $\langle u^*,  -M_{(\cdot, i)} \rangle  =  -(M^* u^*)_i$ for $i = 1, \cdots, n$, we can see that
  \begin{equation}  \label{equiv:WprimeSetVSN123}
   \begin{aligned}
      &   \begin{cases}
  \langle u^*, E_{(\cdot, i)} \rangle =0, i \in I_1  , & \langle u^*, M_{(\cdot, j)} \rangle =0, j\in I_2 \\
  \langle u^*, E_{(\cdot, i)} \rangle \leq 0, i \in \overline{I_1}  , & \langle u^*, M_{(\cdot, j)} \rangle \geq 0, j\in \overline{I_2}
  \end{cases} \ \Longleftrightarrow  \\
  & \hskip5cm \begin{cases}
                                 (u^*_i, -(M^* u^*)_i) \in \{0\} \times \R_- , & \ i \in I_1,\\
                                 (u^*_i, -(M^* u^*)_i )\in \R_- \times \{0\}, & \  i \in I_2, \\
                                 (u^*_i, -(M^* u^*)_i )\in \R^2_- ,  & \  i \in I_3.
                               \end{cases}
    \end{aligned} 
  \end{equation}
  Together with \cref{eqn:LCP-domNC}, it means that
$$ u^* \in N_{\dom S}(\bar{q}) \Longleftrightarrow (u^*, -M^*u^*) \in W'(I_1(\bar{q}, \bar{x}), I_2(\bar{q}, \bar{x}), I_3(\bar{q}, \bar{x})).$$  
Thus we arrive at the sufficient and necessary condition \cref{eqn:sufnecsinglepointwithW}.
\end{proof}

\begin{remark}
  When $\bar{q} \in \inte \dom S, \ N_{\dom S}(\bar{q}) = \{0\}$, the criterion \cref{eqn:sufnecsinglepointwithW}  reduces to
  $$(u^*, -M^* u^*) \in W \left(I_1(\bar{q},\bar{x}),I_2(\bar{q},\bar{x}),I_3(\bar{q},\bar{x})\right)\Longrightarrow u^* =0,$$ 
  which is equivalent to the sufficient and necessary condition for the Lipschitz-like property of $S$ in \cite[Theorem 4]{dontchev1996}. 
\end{remark}

Next we consider the graphical modulus $\lip_{\dom S} S$ based on equality \cref{eqn:LCP-projcode} and give further simplification.
        \begin{corollary}
With the setting of \cref{thm:LCP-sufneccond}, when \cref{eqn:sufnecsinglepointwithW} holds, 
the graphical modulus of $S$ relative to $\dom S$ at $\bar{q}$ for $\bar{x}$ is given by
  \begin{equation}\label{eqn:LCP-lipXS2}
    {\lip}_{\dom S} S(\bar{q}, \bar{x}) = \sup_{(I_1, I_2, I_3) \in \mathcal{I}(\bar{q}, \bar{x})} \  \sup_{u^* \in U(I_1, I_2, I_3)} \frac{d(u^*, N(I_1, I_2, I_3))}{\left\| \left( \begin{matrix}     M^*_{(I_2,\cdot)}  u^* \\ \left( - M^*_{(I_3,\cdot)}  u^* \right)_+\end{matrix}  \right)  \right\| },
  \end{equation}
    where 
    \begin{equation*}\label{Udef}
    U(I_1, I_2, I_3) = \left\{ u^* \in \R^n \left| \begin{matrix}
                                            u^*_i= 0, &   i\in I_1  \\
                                            u^*_i \in \R, & i\in I_2 \\
                                            u^*_i \in \R_-, & i\in I_3
                                                 \end{matrix} \right. \right\},
  \end{equation*} 
  $M^*_{(I_2,\cdot)}$ means the matrix with the rows of $M^*$ indexed by $I_2$ and for a vector $y$, $(y)^+$  means the vector $\left( (y_i)^+ \right)$. 
\end{corollary}
\begin{proof}
  By \cref{prop:AVI-ProjCde-Expression}, 
    \begin{align}
         &  ~~ D^*_{\dom S} S\left( \bar{q}  \mid \bar{x} \right) (y^*) \nonumber\\
    = &\displaystyle \bigcup_{(q,x) \in \gph S\cap \B_{\varepsilon}(\bar{q}, \bar{x})} \bigg\{ \left. {\rm proj}_{T_{dom S}(q)} (u^*) \ \right| \ y^* = -M^* u^*-v^*,  \nonumber\\[-0.8cm]
    &\hskip6cm \left( v^*, -u^*\right) \in N_{\gph N_{\R^m_+}}(x, -Mx-q)  \bigg\} \nonumber\\
    = &\displaystyle   \bigcup_{(q,x) \in \gph S\cap \B_{\varepsilon}(\bar{q}, \bar{x})} \bigg\{ \left. {\rm proj}_{T_{dom S}(q)} (u^*) \ \right| \ y^* = -M^* u^* - v^*,\nonumber\\[-0.8cm]
    &\hskip6cm  \left(u^*,v^*\right) \in W\left(I_1(q,x), I_2(q,x), I_3(q,x) \right)   \bigg\}\nonumber\\
    = &\displaystyle   \bigcup_{\left(I_1, I_2, I_3\right) \in \mathcal{I}(\bar{q}, \bar{x})}  \big\{  {\rm proj}_{T(I_1, I_2,I_3) }(u^*) \mid  \exists u^* \text{ s.t. } \left( u^*, -y^*-M^*u^* \right) \in W(I_1, I_2 ,I_3 )\big\} \label{eqn:LCP-projcode},
    \end{align}
    where the second equality can be derived via \cref{equiv:WidxVSNgphNR+} and the third one is a combination of \cref{prop:LCP-domTCNC} and \cref{thm:LCP-NBHslices}.

    Note that by the polar relation between $N(I_1, I_2, I_3)$ and $T(I_1, I_2, I_3)$ and their convexity, we have
    \[\|{\rm proj}_{T(I_1, I_2,I_3) }(u^*) \| = d(u^*, N(I_1, I_2, I_3)), \ \mbox{ for any } u^*\in \R^n.\] 
    By \cref{thm-genMordcri} and equality \cref{eqn:LCP-projcode} we can see that
         \begin{equation}\label{eqn:LCP-lipXS1}
        {\lip}_{\dom S}S(\bar{q}, \bar{x}) = \sup_{(I_1, I_2, I_3) \in \mathcal{I}(\bar{q}, \bar{x})} \  \sup_{(u^*,v^*)  \in W(I_1, I_2, I_3)} \frac{d(u^*, N(I_1, I_2, I_3))}{\left\|M^* u^* + v^*\right\|}.
        \end{equation}
    Here we further simplify \cref{eqn:LCP-lipXS1} by getting rid of $v^*$.
The second supremum in \cref{eqn:LCP-lipXS1} is equivalent to first maximizing the fractional relative to $v^*$ and then relative to $u^*$. Therefore we simplify the part of maximizing the fractional relative to $v^*$, which is equivalent to minimizing $\|M^*u^* +v^* \|$, i.e., $\displaystyle \min_{v^*_i } | (M^* u^*)_i + v^*_i |$ for each $i \in I$. 
Note that $(M^* u^*)_i  = M^*_{(i,\cdot)} u^*$.
By \cref{eqn:LCP-def-W}, we have the following three cases: 
  \begin{enumerate}
    \item $i\in I_1:\ (u^*_i, v^*_i ) \in \{0\} \times \R$. $\displaystyle \min_{v^*_i \in \R} | M^*_{(i,\cdot)} u^* + v^*_i| = 0$ by taking $v^*_i = -M^*_{(i,\cdot)} u^*$.
    \item $i\in I_2:\ (u^*_i, v^*_i ) \in \R \times \{0\}$.  $\displaystyle \min_{v^*_i =0 }|M^*_{(i,\cdot)} u^* + v^*_i| = | M^*_{(i,\cdot)} u^* |$.
    \item $i\in I_3:\ (u^*_i, v^*_i ) \in \Omega $. 
      Recall that $\Omega= (\{0\} \times \R) \cup (\R\times \{0\})\cup \R^2_-$. Next we define the following index subsets of $I_3$ considering the possibilities of the values of $u^*_i$ and $v^*_i$ 
      respectively:
        \begin{eqnarray*}
          I_{31}&:=& \left\{ i\in I_3 \ | \  (u^*_i, v^*_i ) \in \{0\} \times \R \right\}; \\
          I_{32}&:=& \left\{ i\in I_3 \ | \  (u^*_i, v^*_i ) \in \R \times \{0\} \right\}; \\
          I_{33}&:=&\left\{ i\in I_3 \ | \  (u^*_i, v^*_i ) \in \R^2_- \right\}.
        \end{eqnarray*}
    Thus we have the three subcases as follows:    
    \begin{enumerate}
      \item $i\in I_{31}$: we have $ (u^*_i, v^*_i ) \in  \{0\} \times \R$, thus $\displaystyle \min_{v^*_i \in \R} | M^*_{(i,\cdot)} u^* + v^*_i| = 0$. 
      \item $i\in I_{32}$: we have $ (u^*_i, v^*_i ) \in \R \times \{0\} $, thus $\displaystyle \min_{v^*_i =0} | M^*_{(i,\cdot)} u^* + v^*_i| = | M^*_{(i,\cdot)} u^* |$.
      \item $i\in I_{33}$: we have $ (u^*_i, v^*_i ) \in \R^2_-$.  Here further dividing is required depending on the sign of the value of $M^*_{(i,\cdot)} u^*$:
          \begin{enumerate}
            \item $M^*_{(i,\cdot)} u^* \in \R_+$. Then $\displaystyle \min_{v^*_i \in \R_-} | M^*_{(i,\cdot)} u^* + v^*_i| = 0$ by taking $v^*_i = -M^*_{(i,\cdot)} u^*$.
            \item $M^*_{(i,\cdot)} u^* \in \R_-$. Then $\displaystyle \min_{v^*_i \in \R_-} | M^*_{(i,\cdot)} u^* + v^*_i| = |M^*_{(i,\cdot)} u^* |.$
          \end{enumerate}
    \end{enumerate}
    \end{enumerate}
   Take $I'_1 = I_1\cup I_{31}$, $I'_2 = I_2 \cup I_{32}$ and $I'_3 = I_3 \backslash (I_{31}\cup I_{32})  = I_{33}$. Then we have $(I'_1, I'_2, I'_3) \in \mathcal{I}(\bar{q}, \bar{x})$ as well by definition of $\mathcal{I}(\bar{q},\bar{x})$ \cref{eqn:LCP-NBHidxcomb}. Besides, by \cref{eqn:LCP-domNC}
    we have $N(I'_1,I'_2,I'_3) \subseteq N(I_1, I_2, I_3)$. Thus for any $u^* \in \R^n,$
    $$d(u^*, N(I'_1, I'_2, I'_3)) \geq d(u^*, N(I_1, I_2, I_3)).$$
    Note that by previous analysis on the subcases, the selection on $u^*$ remains the same, i.e., 
    $$u^*_i \in \begin{cases}\{0\}, \ &  i\in I_1 \cup I_{31}  = I'_1\\ \R \ &  i\in I_2 \cup I_{32}  = I'_2 \\ \R_- , \ & i \in I_{33} = I'_3 \end{cases}.$$
    Then we have 
    $$ 
    \frac{d(u^*, N(I_1, I_2, I_3))}{\left\| \left( \begin{matrix}     M^*_{(I_2\cup I_{32},\cdot)}  u^* \\ \left( - M^*_{(I_{33},\cdot)}  u^* \right)_+\end{matrix} \right)  \right\| } \leq \frac{d(u^*, N(I'_1, I'_2, I'_3))}{\left\| \left( \begin{matrix}     M^*_{(I'_2,\cdot)}  u^* \\ \left( - M^*_{(I'_3,\cdot)}  u^* \right)_+\end{matrix}   \right)\right\|}.
    $$
      Then the calculation of cases 3(a), 3(b) for the modulus \cref{eqn:LCP-lipXS1} is redundant when the index combination $(I_1,I_2,I_3)$ runs over all possible elements in $\mathcal{I}(\bar{q},\bar{x})$. 
    So we can consider $I_3$ with the case 3(c) only when the index combination $(I_1, I_2, I_3)$ is fixed. 
    Therefore minimizing $\|M^*u^* +v^* \|$ relative to $v^*$ can be achieved by taking the sum of the minimal values of case 1, case 2 and case 3(c). The combined value is
    $$\sqrt{\sum_{i \in I_2} | M^*_{(i,\cdot)} u^* |^2 + \sum_{i\in I_3} \left( \max (0, -M^*_{(i,\cdot)}  u^* ) \right)^2} = \left\|\begin{matrix}     M^*_{(I_2,\cdot)}  u^* \\ \left( - M^*_{(I_3,\cdot)}  u^* \right)_+\end{matrix}  \right\|$$
    and the corresponding parameter set of $u^*$ is just $U(I_1, I_2, I_3)$. The proof is complete.
    \end{proof}

\section{Conclusions}\label{sect:conclu}
In this paper, we presented several upper estimates of projectional coderivative of the solution mapping for parametric systems. We compared two sufficient conditions that were given by projectional coderivative and directional limiting coderivative respectively. We developed a sufficient condition of the relative Lipschitz-like property of the solution mapping for affine variational inequalities as a generalized critical face condition. For a linear complementarity problem with $Q_0$ matrix, we obtained a sufficient and necessary condition for its solution mapping to have Lipschitz-like property relative to its domain. 

\section*{Acknowledgments} 
The authors would like to thank Li Minghua and Meng Kaiwen for their useful discussions and comments during the development of this project.


\bibliographystyle{siamplain}
\bibliography{References}

\begin{thebibliography}{10}

\bibitem{aubin1984lipschitz}
{\sc J.-P. Aubin}, {\em Lipschitz behavior of solutions to convex minimization
  problems}, Math. Oper. Res., 9 (1984), pp.~87--111.

\bibitem{Benko2020}
{\sc M.~Benko, H.~Gfrerer, and J.~V. Outrata}, {\em Stability analysis for
  parameterized variational systems with implicit constraints}, Set-Valued Var.
  Anal., 28 (2020), pp.~167--193.

\bibitem{bickeletal2009}
{\sc P.~Bickel, Y.~Ritov, and A.~Tsybakov}, {\em Simultaneous analysis of lasso
  and dantzig selector}, Annals of Statistics, 37 (2009), p.~1705–1732.

\bibitem{bonnans2013perturbation}
{\sc J.~F. Bonnans and A.~Shapiro}, {\em Perturbation Analysis of Optimization
  Problems}, Springer Science \& Business Media, 2013.

\bibitem{candestao2005}
{\sc E.~J. Cand`es and T.~Tao}, {\em Decoding by linear programming}, IEEE
  Transactions on Information Theory, 51 (2005), p.~4203–4215.

\bibitem{canovas2010variational}
{\sc M.~J. C{\'a}novas, M.~A. L{\'o}pez, B.~S. Mordukhovich, and J.~Parra},
  {\em Variational analysis in semi-infinite and infinite programming, {I}:
  Stability of linear inequality systems of feasible solutions}, SIAM J.
  Optim., 20 (2010), pp.~1504--1526.

\bibitem{LCP}
{\sc R.~W. Cottle, J.-S. Pang, and R.~E. Stone}, {\em The Linear
  Complementarity Problem}, SIAM, 2009.

\bibitem{dontchev1996}
{\sc A.~L. Dontchev and R.~T. Rockafellar}, {\em Characterizations of strong
  regularity for variational inequalities over polyhedral convex sets}, SIAM J.
  Optim., 6 (1996), pp.~1087--1105.

\bibitem{dontchev2009implicit}
{\sc A.~L. Dontchev and R.~T. Rockafellar}, {\em Implicit Functions and
  Solution Mappings}, Springer New York, 2009.

\bibitem{fang2015minimal}
{\sc Y.~P. Fang, K.~W. Meng, and X.~Q. Yang}, {\em On minimal generators for
  semi-closed polyhedra}, Optimization, 64 (2015), pp.~761--770.

\bibitem{gfrerer2013directional}
{\sc H.~Gfrerer}, {\em On directional metric regularity, subregularity and
  optimality conditions for nonsmooth mathematical programs}, Set-Valued Var.
  Anal., 21 (2013), pp.~151--176.

\bibitem{GfrOut2016}
{\sc H.~Gfrerer and J.~V. Outrata}, {\em On {L}ipschitzian properties of
  implicit multifunctions}, SIAM J. Optim., 26 (2016), pp.~2160--2189.

\bibitem{gfrerer2020aubin}
{\sc H.~Gfrerer and J.~V. Outrata}, {\em On the {A}ubin property of solution
  maps to parameterized variational systems with implicit constraints},
  Optimization, 69 (2020), pp.~1681--1701.

\bibitem{ginchevdirectionally}
{\sc I.~Ginchev and B.~S. Mordukhovich}, {\em On directionally dependent
  subdifferentials}, Compt. rend. Acad. bulg. Sci., 64 (2011), pp.~497--508.

\bibitem{huetal2017}
{\sc Y.~Hu, C.~Li, K.~Meng, J.~Qin, and X.~Yang}, {\em Group sparse
  optimization via $\ell_{p,q}$ regularization}, Journal of Machine Learning
  Research, 18 (2017), p.~1–52.

\bibitem{Huyen2019}
{\sc D.~T.~K. Huyen and J.-C. Yao}, {\em Solution stability of a linearly
  perturbed constraint system and applications}, Set-Valued Var. Anal., 27
  (2019), pp.~169--189.

\bibitem{Huyen2016}
{\sc D.~T.~K. Huyen and N.~D. Yen}, {\em Coderivatives and the solution map of
  a linear constraint system}, SIAM J. Optim., 26 (2016), pp.~986--1007.

\bibitem{ioffe2017variational}
{\sc A.~D. Ioffe}, {\em Variational analysis of Regular Mappings: Theory and
  Applications}, Springer, 2017.

\bibitem{klatte2006nonsmooth}
{\sc D.~Klatte and B.~Kummer}, {\em Nonsmooth Equations in Optimization},
  Kluwer Academic Publishers, 2002.

\bibitem{lee2005quadratic}
{\sc G.~M. Lee, N.~N. Tam, and N.~D. Yen}, {\em Quadratic Programming and
  Affine Variational Inequalities: A Qualitative Study}, Springer, 2005.

\bibitem{lee2014coderivatives}
{\sc G.~M. Lee and N.~D. Yen}, {\em Coderivatives of a
  {K}arush--{K}uhn--{T}ucker point set map and applications}, Nonlinear Anal.,
  95 (2014), pp.~191--201.

\bibitem{LevyMord2004}
{\sc A.~B. Levy and B.~S. Mordukhovich}, {\em Coderivatives in parametric
  optimization}, Math. Program., 99 (2004), pp.~311--327.

\bibitem{lu2008variational}
{\sc S.~Lu and S.~M. Robinson}, {\em Variational inequalities over perturbed
  polyhedral convex sets}, Math. Oper. Res., 33 (2008), pp.~689--711.

\bibitem{Meng2020}
{\sc K.~W. Meng, M.~H. Li, W.~F. Yao, and X.~Q. Yang}, {\em Lipschitz-like
  property relative to a set and the generalized {M}ordukhovich criterion},
  Math. Program., 189 (2021), pp.~455--489.

\bibitem{mordukhovich1992sensitivity}
{\sc B.~S. Mordukhovich}, {\em Sensitivity analysis in nonsmooth optimization},
  in Theoretical Aspects of Industrial Design, D.A. Field and V. Komkov, eds.,
  SIAM, Philadelphia, 58 (1992), pp.~32--46.

\bibitem{mordukhovich1994generalized}
{\sc B.~S. Mordukhovich}, {\em Generalized differential calculus for nonsmooth
  and set-valued mappings}, J. Math. Anal. Appl., 183 (1994), pp.~250--288.

\bibitem{mordukhovich2006variational}
{\sc B.~S. Mordukhovich}, {\em Variational Analysis and Generalized
  Differentiation {\normalfont I}: Basic Theory}, Springer Science \& Business
  Media, 2006.

\bibitem{mordukhovich2007coderivative}
{\sc B.~S. Mordukhovich and J.~V. Outrata}, {\em Coderivative analysis of
  quasi-variational inequalities with applications to stability and
  optimization}, SIAM J. Optim., 18 (2007), pp.~389--412.

\bibitem{Robinson1981}
{\sc S.~M. Robinson}, {\em Some continuity properties of polyhedral
  multifunctions}, Math. Program. Stud., 14 (1981), pp.~206--214.

\bibitem{robinson2007solution}
{\sc S.~M. Robinson}, {\em Solution continuity in monotone affine variational
  inequalities}, SIAM J. Optim., 18 (2007), pp.~1046--1060.

\bibitem{CvxAn}
{\sc R.~T. Rockafellar}, {\em Convex Analysis}, Princeton University Press,
  1997.

\bibitem{VaAn}
{\sc R.~T. Rockafellar and R.~J.-B. Wets}, {\em Variational Analysis}, Springer
  Science \& Business Media, 2009.

\bibitem{yang2010structure}
{\sc X.~Q. Yang and N.~D. Yen}, {\em Structure and weak sharp minimum of the
  pareto solution set for piecewise linear multiobjective optimization}, J.
  Optim. Theory Appl., 147 (2010), pp.~113--124.

\end{thebibliography}
\end{document}